\documentclass[10pt]{article}
\usepackage{graphicx}
\usepackage{url}
\usepackage{amsfonts,amsthm,amsmath}
\usepackage{amssymb}
\usepackage{authblk}
\usepackage[top=1.0in, bottom=1.25in, left=1.25in, right=1.25in]{geometry}
\usepackage{xcolor}
\usepackage{array}
\newcolumntype{x}[1]{>{\centering\arraybackslash\hspace{0pt}}p{#1}}
\usepackage[font=normalsize,labelfont={bf}]{caption}

\usepackage[flushleft]{threeparttable}

\newtheorem{Theorem}{Theorem}[section]

\newtheorem{Example}{Example}[section]

\newtheorem{Lemma}[Theorem]{Lemma}
\newtheorem{Corollary}[Theorem]{Corollary}

\newcommand{\zed}{{\ensuremath{\mathbb{Z}}}} 
\newcommand{\eff}{{\ensuremath{\mathbb{F}}}} 

\newcommand{\G}{\ensuremath{\mathcal{G}}}

\newcommand{\A}{\ensuremath{\mathcal{A}}}

\newcommand{\B}{\ensuremath{\mathcal{B}}}

\renewcommand*{\P}{\ensuremath{\mathcal{P}}}

\title{Nestings of BIBDs with block size four}
\author[1]{Marco Buratti}
\affil[1]{Dipartimento di Scienze di Base e Applicate per l'Ingegneria (S.B.A.I.)\\ 
Sapienza Universit\'{a} di Roma\\ Via Antonio Scarpa, 10, Italy}
\author[2]{Donald L.\ Kreher}
\affil[2]{Department of Mathematical Sciences\\
Michigan Technological University\
Houghton, MI 49931-1295, U.S.A.}
\author[3]{Douglas R.\ Stinson\thanks{D.R.\ Stinson's research is supported by  NSERC discovery grant RGPIN-03882.}}
\affil[3]{David R.\ Cheriton School of Computer Science\\University of Waterloo\\ Waterloo ON, N2L 3G1\\Canada}

\date{\today}

\begin{document}

\maketitle

\begin{abstract}
In a nesting of a balanced incomplete block design (or BIBD), we wish to add a point (the \emph{nested point}) to every block of a 
$(v,k,\lambda)$-BIBD in such a way that we end up with a partial $(v,k+1,\lambda+1)$-BIBD. 
In the case where the partial $(v,k+1,\lambda+1)$-BIBD is in fact a $(v,k+1,\lambda+1)$-BIBD, we have a 
\emph{perfect nesting}. We show that a nesting is perfect if and only if $k = 2 \lambda + 1$. 

Perfect nestings were previously known to exist in the case of Steiner triple systems (i.e., $(v,3,1)$-BIBDs) when $v \equiv 1 \bmod 6$, as well as for some symmetric BIBDs. Here we study 
nestings of $(v,4,1)$-BIBDs, which  are not perfect nestings. 
We prove that there is a nested $(v,4,1)$-BIBD if and only if $v \equiv 1 \text{ or } 4 \bmod 12$, $v \geq 13$. This is accomplished by a variety of direct and recursive constructions.
\end{abstract}

\section{Introduction and background}
\label{weak.sec}

Various kinds of nested designs have been studied for many years. One definition of these objects can be found in 
\cite[VI.36]{CD}, where it is required that a balanced incomplete block design (BIBD) with blocks of size $dk$ can be decomposed into $d$ BIBDs with blocks of size $k$. Here we consider a problem motivated by nesting of Steiner triple systems (STS), e.g., as studied in \cite{St85}. 

An STS$(v)$, say $(X, \A)$, can be \emph{nested} if there is a mapping $\phi : \A \rightarrow X$ such that 
$(X, \{ A \cup \{\phi(A)\} : A \in \A \} )$ is a $(v,4,2)$-BIBD. In words, we are adding a fourth point to every block of the STS$(v)$ so the result is a $(v,4,2)$-BIBD. It was shown in \cite{St85} that, for all $v \equiv 1 \bmod 6$, there exists a nested STS$(v)$.

One obvious generalization would be to start with an arbitrary $(v,k,\lambda)$-BIBD and again add a point to each block. 
The mapping $\phi : \A \rightarrow X$ is a \emph{nesting} provided that 
$(X, \{ A \cup \{\phi(A)\} : A \in \A \} )$ is a \emph{partial} $(v,k+1,\lambda+1)$-BIBD. 
If  the partial $(v,k+1,\lambda+1)$-BIBD is in fact a $(v,k+1,\lambda+1)$-BIBD, we have a 
\emph{perfect nesting}.
 
We define a bit of useful  terminology.   The point $\phi(A)$ is called the \emph{nested point} for the block $A$. 
We refer to $A \cup \{ \phi(A) \}$ as an \emph{augmented block}. The multiset $N_{\phi} = \{ \phi(A) : A \in \A\}$ consists of all the nested points, counting multiplicities. The (possibly partial) $(v,k+1,\lambda+1)$-BIBD obtained from a nesting is called the \emph{augmented design}. In the case where the nesting is perfect, we may also call it the \emph{augmented BIBD}.

\begin{Example}
\label{1341.exam}
The base block $B = \{1,2,4,10\}$, when developed modulo $13$, yields a $(13,4,1)$-BIBD. Thus the blocks of the BIBD are
$B, B+1, \dots , B+12$. If we nest $B$ with $0$ and we nest $B+i$ with $i$ for $1 \leq i \leq 12$, we obtain a nesting of the BIBD. 
The resulting partial $(13,5,2)$-BIBD has the following blocks, where the last point in each block is the nested point:
\[
\begin{array}{cccc}
 \{1,2,4,10,0\} \quad \quad  &
 \{2,3,5,11,1\} \quad \quad  & 
  \{3,4,6,12,2\} \quad \quad  & 
 \{4,5,7,0,3\} \\
  \{5,6,8,1,4\} \quad \quad  &
 \{ 6,7,9,2,5\} \quad \quad  & 
  \{ 7,8,10,3,6\} \quad \quad  & 
 \{8,9,11,4,7\} \\
   \{ 9,10,12,5,8\} \quad \quad  &
 \{ 10,11,0,6,9\} \quad \quad  & 
  \{11,12,1,7,10\} \quad \quad  & 
 \{ 12,0,2,8,11\} \\
   \{ 0,1,3,9,12\} \quad \quad  &
\end{array}
\]
\end{Example}

\medskip

We begin by proving a simple necessary condition for a BIBD to be nested. Here and elsewhere, the \emph{replication number} of a 
$(v,k,\lambda)$-BIBD is the integer $r = \lambda (v-1)/(k-1)$. Every point in the BIBD occurs in exactly $r$ blocks. The number of blocks in the BIBD is denoted by $b$, where $b = vr/k = \lambda v (v-1)/(k(k-1))$.

\begin{Lemma} 
\label{bound.lem} Suppose that a $(v,k,\lambda)$-BIBD has a nesting. Then $k \geq 2 \lambda + 1$.
\end{Lemma}

\begin{proof}
Suppose $(X, \A)$ is a $(v,k,\lambda)$-BIBD and $\phi$ is a nesting. $(X, \{ A \cup \{\phi(A)\} : A \in \A \} )$ is a partial $(v,k+1,\lambda+1)$-BIBD. In the augmented design, the number of pairs that occur in the $b$ blocks is
$b\binom{k+1}{2}$. Since every pair occurs at most $\lambda+1$ times, we have
\begin{equation}
\label{pairs.eq}
 b\binom{k+1}{2} \leq (\lambda+1)\binom{v}{2},
\end{equation}
or
\[ b(k^2+k) \leq (\lambda+1)(v^2-v).\]
Since $b = \lambda v (v-1)/(k(k-1))$, we have
\[ \frac{\lambda (v^2-v)(k^2+k)}{k^2-k} \leq (\lambda+1)(v^2-v).\]
Hence,
\[ \lambda (k+1) \leq (\lambda+1)(k-1).\]
This simplifies to yield the desired inequality.
\end{proof}


In the next theorem, we characterize perfect nestings in terms of the parameters of the BIBD.
\begin{Theorem}
\label{equiv.thm}
Suppose a $(v,k,\lambda)$-BIBD has a nesting. The following are equivalent:
\begin{enumerate}
\item $k = 2 \lambda + 1$, 
\item $v =  2r+1$, and
\item the nesting is perfect.
\end{enumerate}
\end{Theorem}

\begin{proof}
Conditions  1.\ and 2. are equivalent because $\lambda(v-1) = r(k-1)$ in a BIBD.
The equivalence of 1.\ and 3.\ follows easily from the proof of Lemma \ref{bound.lem}. Since we have equality in (\ref{pairs.eq}), it must be the case that every pair of points occurs exactly $\lambda +1$ times in the augmented design. Hence the nesting is perfect.
\end{proof}

\begin{Lemma} 
\label{2k+1.lem}
If a $(v,k,\lambda)$-BIBD has a perfect nesting, then $v \equiv 1 \bmod 2k$.
\end{Lemma}
\begin{proof}
Every point in the augmented BIBD occurs in $r'$ blocks, where $r' = (\lambda+1)(v-1)/k$.
Also, every point occurs in $r$ blocks in the original $(v,k,\lambda)$-BIBD.
Hence, every point occurs in $N_{\phi}$ (the multiset of nested points)  the same number of times, namely $r' - r$ times. 
Using the fact that $k = 2\lambda + 1$, we have
\begin{align*}
 r' - r &= \frac{(\lambda+1)(v-1)}{k}  - \frac{\lambda (v-1)}{k-1}\\
 &= \frac{(k+1)(v-1)}{2k}  - \frac{(k-1) (v-1)}{2(k-1)}\\
 &= \frac{(k+1)(v-1)}{2k}  - \frac{k(v-1)}{2k}\\
 &= \frac{v-1}{2k}.
\end{align*}
Since $r' - r$ is an integer, we have  $v \equiv 1 \bmod 2k$.
\end{proof}


Symmetric $(v,k,\lambda)$-BIBDs having a perfect nesting have been characterized by \'{O} Cathain in \cite{POC}. 
First, we use Theorem \ref{equiv.thm} to determine the possible parameters of a $(v,k,\lambda)$-SBIBD 
that could be perfectly nested. We have $k = r = (v-1)/2$ and $\lambda = (k-1)/2$, so the parameters are those of a $(4t-1,2t-1,t-1)$-SBIBD, as was shown in \cite{POC}. The following characterization is also shown in \cite{POC}.

\begin{Theorem}[\'{O} Cathain \cite{POC}]
\label{SBIBD.thm}
There exists a $(4t-1,2t-1,t-1)$-SBIBD that can be perfectly nested 
if and only if a skew-Hadamard matrix of order $4t$ exists.
\end{Theorem}

At the present time, skew-Hadamard matrices of order $4t$ are known to exist for all positive integers $t < 89$.

\medskip


The case $k=3$, $\lambda = 1$ is the Steiner triple system case. It was shown in \cite{St85} that there is a perfect  nesting of an STS$(v)$ (i.e., a nesting into a $(v,4,2)$-BIBD) for all $v \equiv 1 \bmod{6}$. These are the same as Banff designs (see \cite{BMNR}) with $k=3$ that have exact colourings. 

Nestings have also been considered  from another point of view.
Harmonious colourings of Levi graphs have been studied in a few recent papers \cite{AGMMRRT,PBOM,BMNR}. Connections with nestings were first observed in \cite{BMNR}. We briefly summarize these connections now.
First, we  recall several definitions from \cite{AGMMRRT,PBOM,BMNR}. 
A \emph{harmonious colouring} of a graph $G$ is a proper vertex colouring of $G$ in which no two edges receive the same pair of colours. 
The \emph{harmonious chromatic number} of a graph $G$, denoted by $h(G)$, is the minimum integer $k$ such that $G$ has a harmonious colouring using $k$ colours. 

Let $(X, \A)$ be a $(v,k,\lambda)$-BIBD.  The \emph{Levi graph} of $(X, \A)$, denoted $\mathcal{L}(X, \A)$,  is the bipartite point-vs-block incidence graph of $(X, \A)$. Thus $\mathcal{L}(X, \A)$ has vertices $X \cup \A$, and $xA$ is an edge if and only if $x \in A$. It is easy to see that $h(\mathcal{L}(X, \A)) \geq v$. A  \emph{Banff design}, as defined in \cite{BMNR},  is a $(v,k,\lambda)$-BIBD such that $h(\mathcal{L}(X, \A)) = v$. A Banff design has an \emph{exact colouring} if every pair of colours occurs in exactly one edge.

The following lemma is remarked in \cite{BMNR}. 

\begin{Lemma}
A $(v,k,\lambda)$-BIBD having a perfect nesting is equivalent to a $(v,k,\lambda)$-Banff design having an exact colouring.
\end{Lemma}

The main problem studied in the two papers \cite{AGMMRRT,PBOM} can be viewed as a generalized version of nesting $(v,2,1)$-BIBDs (i.e., complete graphs $K_v$).  Lemma \ref{bound.lem} asserts that a $(v,2,1)$-BIBD cannot be nested.  
The problem addressed in \cite{AGMMRRT,PBOM} allows the number of points to be increased when the design is nested. Nestings are found in these papers that minimize the number of ``new'' points used.

\subsection{Our contributions}

In this paper, we completely determine the spectrum of nested $(v,4,1)$-BIBDs. We prove that there is a nested $(v,4,1)$-BIBD if and only if $v \equiv 1 \text{ or } 4 \bmod 12$, $v \geq 13$.
Of course, nestings of $(v,4,1)$-BIBDs are not perfect nestings because 
$2\lambda + 1 = 3 < 4 = k$.  

First, in Section \ref{nestGDD.sec} we introduce the related concept of \emph{nested group-divisible designs}. Nestings of $(v,4,1)$-BIBDs with 
$v \equiv 1 \bmod 12$ are studied in Section \ref{1mod12.sec}, and those with $v \equiv 4 \bmod 12$ are studied in Section \ref{4mod12.sec}.
Finally, Section \ref{summary.sec} is a brief summary, where we also mention possible future work.

To prove our results, we use a variety of direct and recursive constructions in the rest of the paper. Direct constructions include difference techniques (often based on finite fields) and computer constructions involving backtracking algorithms. Recursive constructions employ pairwise balanced designs, group-divisible designs and difference matrices in various ways.

\section{Nested group-divisible designs}
\label{nestGDD.sec}

The proof of the existence of nested STS$(v)$ given in  \cite{St85} made essential use of \emph{nested group-divisible designs}. This is also a useful approach in studying the existence of other classes of nested $(v,k,\lambda)$-BIBDs, and these designs are also of interest in their own right. We review some standard definitions first.  

A \emph{$(k,\lambda)$-group-divisible design} (or \emph{GDD}) of type $t^u$ is a triple $(X,\G,\A)$ that satisfies the following properties:
\begin{enumerate}
\item $X$ is a set of $tu$ \emph{points}.
\item $\G$ is a partition of $X$ into $u$ \emph{groups} of size $t$, say $\G = \{G_1, \dots , G_u\}$.
\item $\A$ consists of a multiset of \emph{blocks} of size $k$ such that 
\begin{enumerate}
\item $|G_i \cap A| \leq 1$ for $1 \leq i \leq u$ and for all $A \in \A$,
\item every pair of points from different groups is contained in exactly $\lambda$ blocks.
\end{enumerate}
\end{enumerate}
A \emph{partial GDD} is the same as a GDD, except that every pair of points from distinct groups is contained in
\emph{at most} $\lambda$ blocks (pairs of points from the same group never occur in the same block).

We want to study nestings of GDDs. 
A $(k,\lambda)$-GDD of type $t^u$ can be \emph{nested} if there is a mapping 
$\phi : \A \rightarrow X$ such that
$(X,\G,\B)$  is a partial $(k+1,\lambda + 1)$-GDD of type $t^u$, where
\[ \B = \{ A \cup \{\phi(A)\} : A \in \A \} .\]
If the augmented partial GDD is actually a GDD (i.e., a $(k+1,\lambda + 1)$-GDD of type $t^u$), then we say that the nesting of the GDD is 
\emph{perfect}. 

We will prove next that $k = 2\lambda +1$ is a necessary condition for such a perfect nesting of a GDD to exist (this is analogous to the case of perfect nestings of BIBDs).

It is easy to see that every point in a $(k,\lambda)$-GDD of type $t^u$ occurs in exactly $r$ blocks, where
\[ r = \frac{\lambda t(u-1)}{k-1}.\]
Also, the number of blocks is $b$,
where
\[ b 
=  \frac{\lambda t^2 u(u-1)}{k(k-1)}.\]
Of course $r$ and $b$ must be integers, so 
\begin{equation}
\label{cong1.eq} \lambda t(u-1) \equiv 0 \bmod (k-1)
\end{equation}
and \begin{equation}
\label{cong2.eq} \lambda t^2 u(u-1) \equiv 0 \bmod (k(k-1)).\end{equation}
A perfect nesting of a 
$(k,\lambda)$-GDD of type $t^u$ yields a
$(k+1,\lambda + 1)$-GDD of type $t^u$.
Hence, the following additional necessary conditions for a perfect nesting (analogous to (\ref{cong1.eq}) and (\ref{cong2.eq})) must be satisfied:
\begin{equation}
\label{cong3.eq} (\lambda +1) t(u-1) \equiv 0 \bmod k
\end{equation}
and \begin{equation}
\label{cong4.eq} (\lambda +1) t^2 u(u-1) \equiv 0 \bmod ((k+1)k).\end{equation}

The number of blocks is not changed as a result of the nesting, so it must be the case that
\[ \frac{\lambda t^2 u(u-1)}{k(k-1)} = \frac{(\lambda+1) t^2 u(u-1)}{k(k+1)}
.\]
Simplifying, we have
\[ k = 2 \lambda+1
.\]

Finally, in a perfect nesting,
every point occurs in 
\[ r' = \frac{(\lambda+1) t(u-1)}{k} \] blocks in the $(k+1,\lambda + 1)$-GDD of type $t^u$.
So every point occurs as a nested point $r' - r$ times.
Using the fact that $k = 2\lambda + 1$, we have, by the same kinds of calculations that were used in the proof of Lemma \ref{2k+1.lem},
\begin{align*}
 r' - r &= \frac{(\lambda+1)t(u-1)}{k}  - \frac{\lambda t(u-1)}{k-1}\\
 &= \frac{t(u-1)}{2k}.
\end{align*}
Since $r'-r$ is an integer, we have 
\begin{equation}
\label{cong5.eq} t(u-1) \equiv 0 \bmod (2k).
\end{equation}


It is also obvious that $u \geq k+1$, so we have proven the following necessary conditions for existence of a perfectly nested $(k,\lambda)$-GDD of type $t^u$.

    \begin{Theorem} There exists a perfectly nested $(k,\lambda)$-GDD of type $t^u$ only if $k = 2 \lambda +1$, conditions \textup{(\ref{cong1.eq})} and  \textup{(\ref{cong3.eq})}--\textup{(\ref{cong5.eq})} are satisifed, and $u \geq k+1$.
\end{Theorem}

Necessary conditions for existence of $(4,1)$-GDDs of type $t^u$ 
 are obtained from (\ref{cong1.eq}) and (\ref{cong2.eq}):
\begin{equation}
\label{cong4,1.eq} t(u-1) \equiv 0 \bmod 3 \quad \text{and} \quad t^2u(u-1) \equiv 0 \bmod 12.\end{equation}
We  will consider nestings of $(4,1)$-GDDs of type $t^u$ where $t=3$. 
When $t=3$, the two congruences (\ref{cong4,1.eq}) are equivalent to $u \equiv 0,1 \bmod 4$. 

In order to construct a nested GDD, we could try to find an appropriate set of base blocks 
in a group $G$ of order $tu$. 
In general, we would seek base blocks for a nested GDD having a specified block size $k$ and index $\lambda$. 

Let $G$ be an additive 
group, let $H$ be a  subgroup of $G$, and let $k\geq 2$ and $\lambda \geq 1$ be integers.
Here $k$ is the block size and $\lambda$ is the index.
A \emph{relative difference family} in $G \setminus H$ (or \emph{RDF}, for short), denoted by
$(G,H,k,\lambda)$-RDF, is a set of $k$-subsets of $G$ (called \emph{base blocks}) 
that satisfy the following property:
\begin{enumerate}
\item[(1)] The base blocks contain every difference in $G \setminus H$ exactly $\lambda$ times, and they contain no difference in $H$. \end{enumerate}

A \emph{Banff relative difference family} in $G \setminus H$ (or \emph{BRDF}, for short), denoted by
$(G,H,k,\lambda)$-BRDF, is a $(G,H,k,\lambda)$-RDF that satisfies 
the following additional two properties:
\begin{enumerate}
\item[(2)] The base blocks are all disjoint from $H$.
\item[(3)] The base blocks and their negatives are pairwise disjoint. 
\end{enumerate}
We will also use the notation $(g,h,k,\lambda)$-BRDF to denote any $(G,H,k,\lambda)$-BRDF in which $|G| = g$ and $|H| = h$. 

Property (1) ensures we get a GDD having index $\lambda$ by developing the base blocks through $G$; the groups of the GDD are $H$ and its cosets. 
Properties (2) and (3) ensure that the GDD can be nested. We simply nest each base block with $0$ and then develop all the base blocks through the group $G$ in order to obtain the augmented GDD. Note that the groups are not nested.

BRDF are obvious generalizations of the Banff difference families that are defined in 
\cite[Defn.\ 3.4]{BMNR}. 
In fact, a Banff difference family in $G$ (which generates a nested BIBD) is the same thing as a BRDF in $G \setminus \{0\}$.

\begin{Example}
{\rm The base block $\{1,2,4,8\}$ is a $(\zed_{15}, \{0,5,10\}, 4,1)$-BRDF. We verify the required three properties:
\begin{enumerate}
\item[(1)] The differences obtained from the base block are $\pm 1, \pm 2, \pm 3, \pm 4, \pm 6, \pm 7$.
\item[(2)] The base block is disjoint from $\{0,5,10\}$.
\item[(3)] The base block and its negative, namely $\{7,11,13,14\}$, are  disjoint. 
\end{enumerate}
Thus the development of the base block yields a $(4,1)$-GDD of type $3^5$. The groups of the GDD are 
$\{0,5,10\}$, $\{1,6,11\}$, $\{2,7,12\}$, $\{3,8,13\}$, $\{4,9,14\}$. If we nest the base block with $0$, then the development of the
 resulting augmented base block of size five is a partial $(5,2)$-GDD of type $3^5$.} 
\end{Example}

\medskip

Property (3) of a BRDF implies that there are no elements of order two in $G \setminus H$, i.e., all the elements of order $2$ must be in the subgroup $H$. This may limit the groups that can possibly give rise to an appropriate set of base blocks. 

A $(G,H,k,\lambda)$-BRDF yields a perfect nesting of the associated GDD when $k = 2\lambda + 1$. In this case, properties (2) and (3) can be replaced by the following property:

\begin{enumerate}
\item[(2$\,'$)] The base blocks and their negatives constitute a partition  of $G \setminus H$.
\end{enumerate}

We will use $(G,H,4,1)$-BRDFs to generate nested $(v,4,1)$-BIBDs and nested $(4,1)$-GDDs of type $3^u$. This is considered further in the next section.



\section{Nesting $(v,4,1)$-BIBDs with $v \equiv 1 \bmod 12$}
\label{1mod12.sec}

The well-known necessary (and sufficient) condition for existence of a $(v,4,1)$-BIBD is $v \equiv 1,4 \bmod 12$.
In this section, we study nested $(v,4,1)$-BIBDs with $v \equiv 1 \bmod 12$.
We also consider GDDs with block size four and $\lambda =1$. First, we establish some useful results on BRDFs.


Buratti {\it et al.}\ \cite{BMNR} have reported the existence of an $(\eff_q,\{0\},4,1)$-BRDF for all primes and prime powers $q \equiv 1 \bmod 12$. They also present a $(\zed_{85},\{0\},4,1)$-BRDF. We list base blocks for $(\zed_v,\{0\},4,1)$-BRDF for
$v = 13,25,37,49$ and $61$ in  Table \ref{BIBDk=4.tab}, which were all found by computer.

\begin{table}
\caption{BRDFs for nested $(v,4,1)$-BIBDs}
\label{BIBDk=4.tab}
\[
\begin{array}{c|c|c}
v & \text{group} & \text{base blocks} \\ \hline
13	& \zed_{13} & \{1,2,4,10\}\\ \hline
25	& \zed_{5} \times \zed_{5} 
	& \{(3,2), (3,3), (4,2), (2,0)\}, \{(1,4), (1,1), (3,4), (4,0)\} \\ \hline
37	& \zed_{37} & \{1,2,4,25\}, \{3,10,22,30\}, \{5,9,14,20\} \\ \hline
49	& \zed_{49} & \{1,2,4,9\},\{3,7,21,32\},\{5,14,24,37\},\{6,18,33,39\} \\ \hline
61	& \zed_{61} & \{1,2,4,8\},\{3,11,16,37\},\{5,15,29,44\},\{6,23,42,51\},\{7,18,30,48\}
\end{array}
\]
\end{table}

Here is a general construction for $(3v,3,4,1)$-BRDF for certain values of $v$.


\begin{Theorem}
\label{3BRDF.thm}
If $v$ is an integer whose maximal prime power factors are all congruent to $1$ modulo $4$,
then there exists a $(3v,3,4,1)$-BRDF.
\end{Theorem}

\begin{proof}
Let $R_v = \prod_{i=1}^j \eff_{q_i}$ be the ring which is the direct product of all the fields $\eff_{q_i}$ with $q_i$ a maximal prime power factor
of $v$;  for instance, $R_{45}=\eff_5\times\eff_9$. 

Let $x_i$ be a primitive fourth root of 1 in each constituent field  $\eff_{q_i}$, and then define $x = (x_1, \dots , x_j)$.
It is clear that the  
 cyclic group $U$ of order four generated by $x$  
acts semiregularly on $R_v \setminus \{0\}$. 
We also note that $x^2 = -1$, so $U =\{1,x,-1,-x\}$.

This means that if $S$ is any complete system 
of representatives for the $U$-orbits on $R_v \setminus \{0\}$, then we have
\begin{equation}\label{SU}
S \otimes U= R_v \setminus \{0\}.
\end{equation}
Note that, for any two sets $A$ and $B$, we are using the notation $A \otimes B$ to denote the multiset $\{ab: a \in A, b \in B\}$. 
When $|A| = 1$, say $A = \{a\}$, we simply write $aB$ for $\{a\} \otimes B$. 

For each $s\in S$, let us consider the following 4-subset of $\zed_3\times R_v$:
$$B_s=\{(1,s),(1,-s),(2,sx),(2,-sx)\}.$$
The differences arising from pairs of elements of $B_s$ yields the set 
$$\Delta B_s = ( \{0\}\times\{\pm2s,\pm2sx\} ) \ \cup \ (\{1,2\}\times\{\pm s(x-1),\pm s(x+1)\}).$$
Since $x^2=-1$, we have $x-1=x+x^2=x(x+1)$. Thus, replacing $x-1$ with $x(x+1)$ in the above formula, we can write:
$$\Delta B_s=(\{0\}\times 2sU) \ \cup \ (\{1,2\}\times (x+1)sU).$$
Consider the family ${\mathcal F}=\{B_s \ : \ s\in S\}$. The differences arising from the sets in ${\mathcal F}$ are
\begin{align*}
\Delta{\mathcal F} & =\bigcup_{s\in S}\Delta B_s\\
&=\left( \{0\}\times\bigcup_{s\in S} 2sU \right) \ \cup \ \left( \{1,2\}\times \bigcup_{s\in S} (x+1)sU \right)\\
&=(\{0\}\times 2(S \otimes U)) \ \cup \ (\{1,2\}\times(x+1)(S\otimes U)).
\end{align*}
Then, taking into account (\ref{SU}), we have 
\begin{align*}
\Delta{\mathcal F} & = (\{0\}\times (R_v \setminus \{0\})) \ \cup \ ( \{1,2\}\times (R_v \setminus \{0\})) \\
&=(\zed_3\times R_v)\setminus(\zed_3\times\{0\}),
\end{align*}
which means that $\mathcal F$ is a $(\zed_3\times R_v,\zed_3\times\{0\},4,1)$-RDF.

Now note that we have $$B_s \ \cup \ -B_s=\{1,2\}\times sU,$$
so that, again taking  (\ref{SU})  into account,  the union of the base blocks of $\mathcal F$ and their negatives is 
$\{1,2\}\times (R_v \setminus \{0\})$, which is a set, not a multiset! This guarantees that $\mathcal F$ has the ``Banff properties"
(i.e., properties (2) and (3) in the definition of a BRDF)) and hence $\mathcal F$ is a $(3v,3,4,1)$-BRDF.
\end{proof}

We note that when $v$ is a prime, the construction given above is a slight variation of the old construction
for a $1$-rotational $(3v,4,1)$-BIBD given by E.H. Moore \cite{Moore}.

\medskip

The smallest value $v \equiv 1 \bmod 4$ that is not covered by Theorem \ref{3BRDF.thm} is $v = 21$. Here we have a BRDF that was found by computer.

\begin{Example}
{\rm We present base blocks for a $(63,3,4,1)$-BRDF in $\zed_{63}$:
\[ \{1,2,4,8\},\{3,11,16,44\},\{5,14,25,43\},\{6,18,33,50\},\{7,17,31,54\}.\]}
\end{Example}

\medskip

We will now prove that nested $(v,4,1)$-BIBDs exist for  all  $v \equiv 1 \bmod 12$. 
We adapt the recursive techniques first described in \cite{St85}. We make use of group-divisible designs with multiple block sizes and group sizes (however, for simplicity, we restrict the definition to $\lambda = 1$). Let $K$ and $L$ be sets of positive integers. A \emph{$(K,1)$-GDD with group sizes in $L$}  is a triple $(X,\G,\A)$ that satisfies the following properties:
\begin{enumerate}
\item $X$ is a set of  \emph{points}.
\item $\G$ is a partition of $X$ into \emph{groups}, say $\G = \{G_1, \dots , G_u\}$, such that $|G| \in L$ for all $G \in \G$.
\item $\A$ consists of a set of \emph{blocks} such that 
\begin{enumerate}
\item $|G \cap A| \leq 1$ for all $G \in \G$ and for all $A \in \A$,
\item every pair of points from different groups is contained in exactly one block, and
\item $|A| \in K$ for all $A \in \A$.
\end{enumerate}
\end{enumerate}
The \emph{type} of the GDD is the multiset $\{|G|: G \in \G\}$. We usually use an exponential notation to describe types:
type ${t_1}^{u_1} {t_2}^{u_2}  \dots$ denotes $u_i$ occurences of $t_i$ for $i = 1,2, \dots$. 

\medskip

The following theorem is a straightforward adaptation of the main construction from \cite{St85}. 

\begin{Theorem}
\label{recursive.thm}
Suppose $(X,\G,\A)$ is a $(K,1)$-GDD with group sizes in $L$ that satisfies the following properties:
\begin{enumerate}
\item For all $k \in K$, there is a nested $4$-GDD of type $3^k$.
\item For all $\ell \in L$, there is a nested $(3\ell + 1,4,1)$-BIBD.
\end{enumerate}
Then there is a nested $(3|X| + 1,4,1)$-BIBD.
\end{Theorem}

\begin{proof}
Start with the hypothesized GDD $(X,\G,\A)$, give every point weight three and apply Wilson's Fundamental GDD construction \cite[\S IV.2.1, Theorem 2.5]{CD}. Blocks are replaced by nested GDDs whose groups are the three copies of each of the points in the block. The groups are replaced by nested BIBDs that include one ``new'' point along with the copies of all the points in the group.
\end{proof}

We also require the notion of a \emph{pairwise balanced design} (or \emph{PBD}).
A \emph{$(v,K)$-PBD}  is a pair $(X,\A)$ that satisfies the following properties:
\begin{enumerate}
\item $X$ is a set of $v$ \emph{points}.
\item $\A$ consists of a set of \emph{blocks} such that 
\begin{enumerate}
\item every pair of points is contained in exactly one block, and
\item $|A| \in K$ for all $A \in \A$.
\end{enumerate}
\end{enumerate}

%
%
We can now prove the following complete existence result for $v \equiv 1 \bmod 12$.

\begin{Theorem}
\label{1mod12.thm}
For all $v \equiv 1 \bmod 12$, there is a nested $(v,4,1)$-BIBD.
\end{Theorem}

\begin{proof}
It is well-known that there is a $(w,\{5,9\})$-PBD for all 
$w \equiv 1 \bmod 4$, $w \neq 13,17,29,33,113$ (see \cite[\S IV.3.2, Table 3.23]{CD}).
Delete a point $x$ from such a PBD. The blocks that contained $x$ yield the groups (having sizes $4$ and $8$) of a GDD. 
Hence, we obtain a $\{5,9\}$-GDD with groups sizes in $\{4,8\}$, for all 
$w \equiv 0 \bmod 4$, $w \neq 12,16,28,32,112$. 
There is a nested $(13,4,1)$-BIBD and a nested $(25,4,1)$-BIBD, as well as nested $4$-GDDs of types $3^5$ and $3^9$ (see Table \ref{BIBDk=4.tab} and Theorem \ref{3BRDF.thm}). 
Applying Theorem \ref{recursive.thm}, we obtain
a nested $(3w + 1,4,1)$-BIBD for all $w \equiv 0 \bmod 4$, $w \neq 12,16,28,32,112$. 
The results from \cite{BMNR} that were mentioned at the beginning of this section
show that nested $(3w + 1,4,1)$-BIBDs exist for these five values of $w$.

\end{proof}

\section{Nesting $(v,4,1)$-BIBDs with $v \equiv 4 \bmod 12$}
\label{4mod12.sec}

We begin by presenting various small nested $(v,4,1)$-BIBDs for $v \equiv 4 \bmod 12$. 

\begin{Example}
\label{1641.exam}
{\rm A nested $(16,4,1)$-BIBD. We list the 20 blocks in the augmented design (the last point in each block is the nested point):
\[
\begin{array}{cccc}
 \{3, 6, 12, 7, 1\} \quad \quad  &
 \{4, 7, 8, 15, 2\} \quad \quad  & 
  \{5, 8, 9, 16, 1\} \quad \quad  & 
 \{6, 9, 15, 10, 3\} \\
  \{ 7, 16, 11, 13, 3\} \quad \quad  &
 \{ 8, 11, 10, 12, 6\} \quad \quad  & 
  \{  9, 12, 13, 14, 4\} \quad \quad  & 
 \{11, 14, 15, 1, 10\} \\
   \{ 12, 15, 16, 2, 9\} \quad \quad  &
 \{ 16, 1, 10, 3, 4\} \quad \quad  & 
  \{10, 2, 13, 4, 5\} \quad \quad  & 
 \{ 15, 13, 3, 5, 14\} \\
   \{ 16, 4, 14, 6, 7\} \quad \quad  &
 \{10, 14, 5, 7, 8\} \quad \quad  & 
  \{ 13, 6, 1, 8, 15\} \quad \quad  & 
 \{ 1, 7, 2, 9, 13\} \\
   \{14, 2, 8, 3, 12\} \quad \quad  &
 \{3, 9, 4, 11, 8\} \quad \quad  & 
  \{1, 4, 5, 12, 11\} \quad \quad  & 
 \{ 2, 5, 11, 6, 16\} 
\end{array}
\]}
\end{Example}

\begin{Example}
\label{2841.exam}
{\rm A nested $(28,4,1)$-BIBD. We take the points to be 
$({\zed_3})^3 \cup \{ \infty\}$.
We specify three base blocks. The first two base blocks are developed modulo $(3,3,3)$; the last base block is developed modulo
$(-,3,3)$.
The base blocks are as follows (the last point in each block is the nested point):
\[
\begin{array}{ll}
\{(0,1,0), (0,2,0), (1,1,2), (1,2,1), (0,1,2)\} \quad \quad \quad &
\{(0,0,1), (0,0,2), (1,2,2), (1,1,1), (2,0,0)\} \\
\{(0,0,0), (1,0,0), (2,0,0), \infty, (0,1,0)\} &
\end{array}
\]
This augmented design has an automorphism of order $9$, generated by $\{0\} \times \zed_3 \times \zed_3$.}
\end{Example}

\smallskip

The next examples are all defined in cyclic groups $\zed_v$, where $v \equiv 4 \bmod 12$.  The designs are generated from $(v+8)/12$ base blocks. The last base block will be $\{ 0, v/4, v/2,3v/4,x\}$ for some $x$. We take all translates of the first $(v-4)/12$ base blocks, and the first $v/4$ translates of the last base block. 

To construct a nested $(v,4,1)$-BIBD, we could start with a known example of a cyclic $(v,4,1)$-BIBD (which without loss of generality contains $\{ 0, v/4, v/2,3v/4\}$ as a base block) and then find a valid nested point for each base block. This is very quick using backtracking. Basically all that is required is that the differences (plus and minus) that include a nested point are all different and none of them is $0$ or $v/2$.  Note that this does not create a cyclic nested BIBD because we only take the first $v/4$ translates of the last base block.

It was recently shown in \cite{ZFW} that a cyclic $(v,4,1)$-BIBD exists if and only if 
$v \equiv 1 \text{ or } 4 \bmod 12$, $v \neq 16,25$ or $28$. Many of the nested BIBDs we construct use known examples of cyclic $(v,4,1)$-BIBDs as a starting point. These BIBDs can be found in several papers, including 
\cite{B2002,Chang,Chen-Wei}. 

Here is a different (but essentially equivalent) way of looking at the problem.
Clearly we can replace any base block by any translate of it. Suppose we take translates of the first $(v-4)/12$ base blocks (we are just omitting the last base block) in such a way that the nested point in each of these base blocks is $0$. If we delete the nested points from the augmented BIBD, then we have a slight relaxation of a $(v,4,4,1)$-BRDF that we call a \emph{weak $(v,4,4,1)$-BRDF}. 
The only difference from a BRDF is that, in a weak $(v,4,4,1)$-BRDF, base blocks may contain one of the two points $v/4$ or $3v/4$, whereas this is not allowed in a $(v,4,4,1)$-BRDF. That is, the first $(v-4)/12$ base blocks thus obtained satisfy properties (1) and (3) of a $(v,4,4,1)$-BRDF, but they are not required to satisfy property (2).

Thus, given a weak $(v,4,4,1)$-BRDF, we have implicitly defined 0 to be the nested point for each base block. In order to obtain the augmented BIBD, we need to also find a suitable nested point $x$ for the ``short orbit'' $H = \{ 0, v/4, v/2,3v/4\}$ (which is just the subgroup of order $4$).
A point $x$ will be \emph{suitable} if the set 
\[ \pm(H - x)  = \{ \pm x , \pm (x - v/4), \pm (x - v/2), \pm (x - 3v/4) \}  \]
is disjoint from the union of the blocks in the BRDF and their negatives. 

Summarizing, we have the following connections between BRDFs, weak BRDFs, nested GDDs and nested BIBDs.
\begin{itemize}
\item A $(v,4,4,1)$-BRDF generates a nested GDD of type $4^{v/4}$ by nesting every base block with 0 and developing the augmented base blocks through the group $G$.
\item Suppose we have a weak $(v,4,4,1)$-BRDF and we are also able to find a suitable nested point $x$ for the subgroup $H$ of order 4.
Then the development of the augmented base blocks through $G$, together with the first $v/4$ translates of $H \cup \{x\}$, 
form a nested $(v,4,1)$-BIBD.
\end{itemize}

In the following examples, the nested points are the last points in each base block.

\begin{Example}
\label{4041.exam}
%
%
{\rm Suppose we start with the following weak $(40,4,4,1)$-BRDF: 
\[
\begin{array}{ccc}
\{2,1,16,10\} \quad \quad \quad &
\{25,18,6,29\} \quad \quad \quad &
\{12,9,36,14\}. 
\end{array}
\] 
These three blocks do not comprise a $(40,4,4,1)$-BRDF because $10 = v/4$ occurs  in the first base block. 

The point $x = 3$ is suitable, since
\[\pm (H - 3) = \{ \pm 3 , \pm 7, \pm 17, \pm 13 \}\] is disjoint from the three base blocks and their negatives.
It follows that we obtain the following base  blocks for a nested $(40,4,1)$-BIBD (where the last point in each block is the nested point):
\[
\begin{array}{ccc}
\{2,1,16,10,0\} \quad \quad \quad &
\{25,18,6,29,0\} \quad \quad \quad &
\{12,9,36,14,0\}\\
\{0,10,20,30,3\} \quad \quad \quad &
\end{array}
\] 
The first three base blocks are developed modulo $40$, and the last base block is translated by $0, \dots , 9$.

At the present time, we do not have a set of base blocks for a nested $(40,4,1)$-BIBD in which the first three base blocks 
form a BRDF. However, we do have an example of a $(40,4,4,1)$-BRDF:
\[
\begin{array}{ccc} 
\{3,2,17,11\} \quad \quad \quad &
\{1,34,22,5\} \quad \quad \quad &
\{12,9,36,14\}
\end{array}
\] }
\end{Example}

\medskip

The following examples use a presentation similar to Example \ref{4041.exam}.
All of the base blocks (except for the last one) are developed modulo $v$. For the last base block, we just take the first $v/4$ translates. The nested point is always the last point in a base block.

\begin{Example}
\label{5241.exam}
%
%
{\rm Base blocks for a nested $(52,4,1)$-BIBD:
\[
\begin{array}{ccc}
\{41,40,19,46,0\} \quad \quad \quad &
\{17,15,8,50,0\} \quad \quad \quad &
\{27,24,9,38,0\}\\
\{51,47,31,7,0\} \quad \quad \quad &
\{0,13,26,39,3\} \quad \quad \quad &
\end{array}
\]
Here the first four base blocks yield a BRDF, because neither $13$ nor $39$ is  an element of any of the first four base blocks.

We remark that an alternate construction for $v=52$ is given in Example \ref{4x13.exam}.}
\end{Example}

\begin{Example}
\label{6441.exam}
%
%
{\rm Base blocks for a nested $(64,4,1)$-BIBD: 
\[
\begin{array}{ccc}
\{20,21,23,27,0\} \quad \quad \quad &
\{50,55,4,33,0\} \quad \quad \quad &
\{26,34,59,6,0\}\\
\{57,2,12,36,0\} \quad \quad \quad &
\{63,11,25,48,0\} \quad \quad \quad &
\{0,16,32,48,3\}
\end{array}
\]
Here the first five base blocks yield a BRDF.}
\end{Example}

\begin{Example}
\label{7641.exam}
%
%
{\rm Base blocks for a nested $(76,4,1)$-BIBD: 
\[
\begin{array}{ccc}
\{4,5,11,26,0\} \quad \quad \quad &
\{7,9,18,52,0\} \quad \quad \quad &
\{20,23,3,15,0\}\\
\{42,46,74,16,0\} \quad \quad \quad &
\{47,57,8,22,0\} \quad \quad \quad &
\{75,12,35,59,0\}\\
\{0,19,38,57,6\} \quad \quad \quad &
\end{array}
\]
Here the first six base blocks yield a BRDF.}
\end{Example}

\begin{Example}
\label{8841.exam}
%
%
{\rm Base blocks for a nested $(88,4,1)$-BIBD:
\[
\begin{array}{ccc}
\{8,9,13,26,0\} \quad \quad \quad &
\{14,17,29,68,0\} \quad \quad \quad &
\{19,28,64,5,0\}\\
\{40,42,61,72,0\} \quad \quad \quad &
\{51,57,82,4,0\} \quad \quad \quad &
\{63,70,2,30,0\}\\
\{87,7,23,49,0\} \quad \quad \quad &
\{0,22,44,66,10\} \quad \quad \quad &

\end{array}
\]
Here the first seven base blocks yield a BRDF.}
\end{Example}

\begin{Example}
\label{10041.exam}
%
%
{\rm Base blocks for a nested $(100,4,1)$-BIBD: 
\[
\begin{array}{ccc}
\{96,97,99,5,0\} \quad \quad \quad &
\{94,98,14,53,0\} \quad \quad \quad &
\{88,93,19,41,0\}\\
\{62,73,92,24,0\} \quad \quad \quad &
\{70,77,91,33,0\} \quad \quad \quad &
\{29,57,69,42,0\}\\
\{49,84,66,20,0\} \quad \quad \quad &
\{79,56,89,13,0\} \quad \quad \quad &
\{0,25,50,75,10\}\\
\end{array}
\]
Here the first eight base blocks yield a BRDF.

We remark that an alternate construction for $v=100$ is given in Theorem \ref{pp.thm}.}
\end{Example}

\begin{Example}
\label{11241.exam}
%
%
{\rm Base blocks for a nested $(112,4,1)$-BIBD: 
\[
\begin{array}{ccc}
\{111,4,12,21,0\} \quad \quad \quad &
\{110,8,19,38,0\} \quad \quad \quad &
\{107,7,9,54,0\}\\
\{86,101,102,15,0\} \quad \quad \quad &
\{88,106,27,51,0\} \quad \quad \quad &
\{13,33,71,77,0\}\\
\{55,78,82,16,0\} \quad \quad \quad &
\{14,48,83,90,0\} \quad \quad \quad &
\{63,92,95,32,0\}\\
\{0,28,56,84,3\} \quad \quad \quad &
\end{array}
\]
Here the first nine base blocks yield a BRDF.}
\end{Example}

\begin{Example}
\label{12441.exam}
%
%
{\rm Base blocks for a nested $(124,4,1)$-BIBD: 
\[
\begin{array}{ccc}
\{1,2,67,4,0\} \quad \quad \quad &
\{3,8,85,18,0\} \quad \quad \quad &
\{5,30,43,80,0\}\\
\{7,40,48,88,0\} \quad \quad \quad &
\{6,51,107,15,0\} \quad \quad \quad &
\{22,66,38,11,0\}\\
\{75,87,104,53,0\} \quad \quad \quad &
\{78,14,99,92,0\} \quad \quad \quad &
\{45,97,26,115,0\}\\
\{65,69,89,95,0\} \quad \quad \quad &
\{0,31,62,93,10\} \quad \quad \quad &
\end{array}
\]
Here the first ten base blocks yield a BRDF.}
\end{Example}




Several additional small examples of nested $(v,4,1)$-BIBDs (for $v \equiv 4 \bmod 12$) are presented in the Appendix.

\medskip

From Examples \ref{1641.exam}--\ref{12441.exam} and \ref{13641.exam}--\ref{19641.exam}, we have

\begin{Lemma}
\label{mod96.lem}
There exist nested $(v,4,1)$-BIBDs for $v \equiv 4 \bmod 12$, $16 \leq v \leq 196$.
\end{Lemma}

Nested $(4,1)$-GDDs of type $3^u$, $u \equiv 0 \bmod 4$, are useful ingredients in constructing nested $(v,4,1)$-BIBDs with $v \equiv 4 \bmod 12$. We present one such  nested GDD in Example \ref{38GDD.exam}.

\begin{Example}
\label{38GDD.exam}
{\rm A nested $(4,1)$-GDD of type $3^8$.
There are $42$ blocks in this GDD, and $42 = 2  \times 21$, so we want two base blocks developed modulo $21$. But there are $24$ points. So we take the point set of the GDD to be $\zed_{21} \ \cup \ \{\infty_0,\infty_1,\infty_2\}$. 
The groups are the seven cosets of $\{0,7,14\}$ in $\zed_{21}$ and $\{\infty_0,\infty_1,\infty_2\}$.
The blocks are the $\zed_{21}$-orbits of 
$$\{2,4,10,13,{0}\},\quad\{15,16,20,\infty_0,{0}\}$$
where the action of $\zed_{21}$ on $\zed_{21}$ itself is the natural one, whereas the action of $\zed_{21}$
on $\{\infty_0,\infty_1,\infty_2\}$ is defined by $\infty_i+g=\infty_{i+g \bmod 3}$ for every pair $(i,g)\in\{0,1,2\}\times\zed_{21}$.
As usual, the last point in each block is the nested point.}
\end{Example}

\medskip

We will now state some useful corollaries of Theorem \ref{recursive.thm}.
We observe that we can only apply Theorem \ref{recursive.thm} using a GDD having block and group sizes that are congruent to $0$ or $1$ modulo $4$. 

Our main construction makes use of \emph{transversal designs} TD$(k,m)$. A TD$(k,m)$ is a $(k,1)$-GDD of type $m^k$. Thus every block in a transversal design is a transversal of the $k$ groups. It is well-known that a TD$(k,m)$ is equivalent to $k-2$ mutually orthogonal latin squares of order $m$.

\begin{Corollary}
\label{4.cor}
Suppose $m \equiv 0\text{ or } 1 \bmod 4$ and let $0 \leq t \leq m$. Suppose 
there exists a TD$(9,m)$ and a nested $(3t + 1,4,1)$-BIBD.
If $m \equiv 1 \bmod 4$, suppose also that there exists a nested $(3m + 1,4,1)$-BIBD.
Then there exists a
nested $(3(8m+t) + 1,4,1)$-BIBD.
\end{Corollary}

\begin{proof}
Delete $m-t$ points from a group of a TD$(9,m)$. We obtain an $(\{8,9\},1)$-GDD of type $m^8 t^1$.
From Example \ref{38GDD.exam} and Theorem  \ref{3BRDF.thm} with $v=9$, there are nested $(4,1)$-GDDs of types $3^{8}$ and $3^9$ (resp.). 
If $m \equiv 0 \bmod 4$, there is a nested $(3m + 1,4,1)$-BIBD by Theorem \ref{1mod12.thm}.
If $m \equiv 1 \bmod 4$, there is a nested $(3m + 1,4,1)$-BIBD by hypothesis.
Finally, there is a nested $(3t + 1,4,1)$-BIBD by hypothesis.
The desired result is then obtained by applying Theorem \ref{recursive.thm}.
\end{proof}

It will be useful to know when TD$(9,m)$ exist. Results can be found in the \emph{Handbook of Combinatorial Designs} \cite[\S III.3.6, Table 3.87]{CD}.
First we focus on $m \equiv 0 \bmod 4$. We will actually only make use of TD$(9,m)$ with $m \equiv 0 \bmod 8$, since there are almost no ``small'' examples of TD$(9,m)$ with $m \equiv 4 \bmod 8$. On the other hand, TD$(9,m)$ with $m \equiv 0 \bmod 8$ always exist.

\begin{Lemma}\label{TD9.lem} 
If $m \equiv 0 \bmod 8$, then there is a TD$(9,m)$.
\end{Lemma}
\begin{proof}
See \cite[\S III.3.6]{CD}.
\end{proof}

We now prove an easy preliminary bound using Corollary \ref{4.cor}.

\begin{Theorem}
\label{gen.bound}
For all $v \equiv 4 \bmod 12$, $v \geq 1924$, there is a nested $(v,4,1)$-BIBD.
\end{Theorem}

\begin{proof}
Assume $v \equiv 4 \bmod 12$, $v \geq 1924$. Write $v = 192w + u$, where $u \equiv 4 \bmod 12$, $16 \leq u \leq 196$
(observe that these values of $u$ cover all the residue classes modulo $192$ that are congruent to $4$ modulo $12)$.
For each such value of $u$, there is a nested $(u,4,1)$-BIBD from Lemma \ref{mod96.lem}.
We have $v = 3(8m+t) + 1$, where $m = 8w$ and $t =  (u-1)/3$.
Hence there is a TD$(9,m)$ from Lemma \ref{TD9.lem}. 
We have 
\[m = 8w = \frac{v-u}{24} \geq \frac{1924-196}{24} = 72\] and 
\[ t \leq \frac{196-1}{3} = 65.\]
Hence $m \geq t$, so we can apply Corollary \ref{4.cor}. 
The result is a nested $(v,4,1)$-BIBD. \end{proof}

\medskip

Using Corollary \ref{4.cor}, we can also construct nested $(v, 4, 1)$-BIBDs  for many values $v \equiv 4 \mod 12$, $v < 1924$.
For a given value of $m$, we can construct nested $(v,4,1)$-BIBDs 
for an interval of values $v \equiv 4 \bmod 12$, as indicated in Table \ref{intervals.tab}. In every case, we have  
$v = 3(8m+t)+1$, as in the proof of Theorem \ref{gen.bound}. 
Any value of $m \equiv 0 \bmod 8$ can be used. 
On the other hand, when
$m \equiv 1 \bmod 4$, we need to check \cite[\S III.3.6]{CD} to see that a TD$(9,m)$ exists, and we also require a nested $(3m+1,4,1)$-BIBD.
For the values of $m$ considered in  Table \ref{intervals.tab}, existence of the required nested $(3m+1,4,1)$-BIBDs follows from Lemma \ref{mod96.lem}.

\begin{table}[tb]
\caption{Intervals of values $v\equiv 4 \bmod 12$ for which nested $(v,4,1)$-BIBDs exist}
\label{intervals.tab}
\[
\begin{array}{c|c|c}
m & t & v\\ \hline
72 & 5,9, \dots , 65 & 1744, 1756, \dots , 1924\\
65 & 5,9,\dots , 65 & 1576, 1588, \dots , 1756\\
64 & 5,9, \dots , 61 & 1552, 1564, \dots , 1720\\
61 & 5,9,\dots , 61 & 1480, 1492, \dots, 1648\\
56 & 5,9, \dots , 53 & 1360, 1372, \dots , 1504\\
53 & 5,9,\dots , 53 & 1288 , 1300, \dots , 1432\\
49 & 5,9,\dots , 49 & 1192, 1204, \dots , 1324\\
48 & 5,9, \dots , 45 & 1168, 1180, \dots , 1288\\
41 & 5,9, \dots, 41 & 1000, 1012, \dots , 1108\\
40 & 5,9, \dots , 37 & 976, 988, \dots , 1072\\
37 & 5,9,\dots , 37 & 904, 916, \dots , 1000\\
32 & 5,9, \dots , 29 & 784, 796, \dots , 856\\
29 & 5,9,\dots , 29 & 712, 724, \dots , 784\\
25 & 5,9,\dots , 25 & 616, 628, \dots , 676\\
24 & 5,9, \dots , 21 & 592, 604, \dots , 640\\
17 & 5,9,13,17 & 424, 436, 448, 460\\
16 & 5,9,13 & 400, 412, 424\\
13 & 5,9,13 & 328, 340, 352\\
9 & 5,9 & 232, 244\\
8 & 5 & 208
\end{array}
\]
\end{table}

Examining  Table \ref{intervals.tab}, we see that  there are 29 values $v \equiv 4 \bmod 12$, $v \geq 4$, yet to be constructed:

\[
\begin{array}{|r|r|r|r|r|r|r|r|}\hline
220 & 256 & 268 & 280 & 292 & 304 & 316 & 364 \\ \hline
376 & 388 & 472 & 484 & 496 & 508 & 520 & 532 \\\hline
544 & 556 & 568 & 580 &  688 & 700 & 868 & 880\\\hline
 892 & 1120 & 1132 & 1144 & 1156 &&&\\\hline
\end{array}
\]

\medskip

Here is a simple construction that can be used to eliminate several of these exceptions.

\begin{Lemma}
\label{RBIBD.lem}
 Suppose there is a resolvable $(56m+8,8,1)$-BIBD. Let $0 \leq t \leq 8m$ and suppose there exists a nested $(3t+1,4,1)$-BIBD.
Then there exists a nested $(v,4,1)$-BIBD for $v = 168m+3t+25$.
\end{Lemma}

\begin{proof}
The resolvable $(56m+8,8,1)$-BIBD has $8m+1$ parallel classes of blocks, say $\P_i$, $1 \leq i \leq 8m+1$. Add infinite points to the blocks in $t$ of the parallel classes, where $0 \leq t \leq 8m$. That is, for $1 \leq i \leq t$, adjoin a new point $\infty_i$ to every block in $\P_i$. 
The infinite points, namely $\{\infty_i : 1 \leq i \leq t\}$, comprise a group in the GDD. The $7m+1$ blocks in another parallel class, say $\P_{t+1}$, yield the remaining groups. The result is an $(\{8,9\},1)$-GDD of type $8^{7m+1}t^1$.

We have nested $(25,4,1)$-BIBDs (Table \ref{BIBDk=4.tab}) and nested $(4,1)$-GDDs of type $3^{8}$ (Example \ref{38GDD.exam}) and type $3^9$ (Theorem  \ref{3BRDF.thm} with $v=9$). Finally, a nested $(3t+1,4,1)$-BIBD exists by hypothesis. Apply Theorem \ref{recursive.thm} to obtain a nested $(3(56m+8) +3t+1,4,1)$-BIBD.
\end{proof}

We note that if we want to construct a nested $(v,4,1)$-BIBD with $v \equiv 4 \bmod 12$ using Lemma \ref{RBIBD.lem}, then we must take
$t \equiv 1 \bmod 4$, $t \geq 5$. 

\begin{Corollary}
There exist nested $(v,4,1)$-BIBDs for 
\[v \in \{376, 388,544, 556, 568, 580, 880,892,1120, 1132, 1144, 1156\}.\]
\end{Corollary}

\begin{proof}
In this proof, existence of the required resolvable $(56m+8,8,1)$-BIBDs follows from \cite[\S II.7.4, Table 7.41]{CD}.

First, we apply Lemma \ref{RBIBD.lem} with $m = 2$, $t = 5, 9$ to obtain  nested $(376,4,1)$- and $(388,4,1)$-BIBDs. Here we require a resolvable $(120,8,1)$-BIBD and
nested $(16,4,1)$- and $(28,4,1)$-BIBDs as ingredients.

For $v = 544, 556, 568$ and $ 580$, let $m = 3$ and $t = 5,9,13,17$. Here we require a resolvable $(176,8,1)$-BIBD and
nested $(16,4,1)$-, $(28,4,1)$-, $(40,4,1)$- and $(52,4,1)$-BIBDs as ingredients.

For $v = 880, 892$, let $m = 5$ and $t = 5,9$. Here we require a resolvable $(288,8,1)$-BIBD and
nested $(16,4,1)$- and $(28,4,1)$-BIBDs as ingredients.

For $v = 1120, 1132, 1144$ and $1156$, let $m = 6$ and $t = 29, 33, 37, 41$. Here we require a resolvable $(344,8,1)$-BIBD and
nested $(88,4,1)$-, $(100,4,1)$-, $(112,4,1)$- and $(124,4,1)$-BIBD as ingredients.
\end{proof}

We now present a  construction that makes use of difference matrices.
Given an additive 
group $G$ of order $g$ and a positive integer $k$, a 
\emph{$(G,k,1)$ difference matrix} (\emph{DM}, for short) is a 
$k\times g$ matrix $M$ with entries from $G$ such that the difference of any two distinct rows of $M$ is a 
permutation of the elements of $G$. A difference matrix $M$ is \emph{homogeneous}  (\emph{HDM}, for short) if
every row of $M$ is a permutation of $G$.


Actually, for our purposes, we only need homogeneous difference matrices in groups of prime power order, which are easy to construct. Specifically, 
the multiplication table of $\eff_q$ is an $(\eff_q,q,1)$-DM for any prime power $q$. 
Such a difference matrix has a row of $0$'s. Suppose $1 \leq k \leq q-1$. If we delete the row of $0$'s and $q-k - 1$ additional rows,  we get an $(\eff_q,k,1)$-HDM. 

Here is a useful ``product construction'' for BRDF that uses homogeneous difference matrices.

\begin{Lemma}
\label{DF1.lem}
If there exist a  $(G_1,H,k,1)$-BRDF and a $(G_2,k,1)$-HDM, then there exists a  $(G_1\times G_2,H\times G_2,k,1)$-BRDF.
\end{Lemma}

\begin{proof}
For each block $B=\{b_1,b_2,\dots, b_k\}$ of a $(G_1,H,k,1)$-BRDF, say ${\cal F}$, and for each column
$(m_{1c},\dots,m_{kc})^T$ of a $(G_2,k,1)$-HDM, say $M=(m_{rc})$,
let $$B\circ M^c=\{(b_1,m_{1c}),(b_2,m_{2c}),\dots,(b_k,m_{kc})\}.$$
Then define $${\mathcal F}\circ M:=\{B\circ M^c \ : \ B\in{\cal F}, 1\leq c\leq |G_2|\}.$$

First, it is easy to check that
$\mathcal{F}\circ M$ is a $(G_1\times G_2,H\times G_2,k,1)$-RDF. 
Consider a difference $(g_1,g_2)$ where $g_1 \in G_1 \setminus H$ and $g_2 \in G_2$. There is a unique block $B \in \mathcal{F}$ such that there exists
$b_i, b_j \in B$ with $b_i - b_j = g_1$. There is also a unique column $c$ of $M$ such that $m_{ic} - m_{jc} = g_2$, because $M$ is a difference matrix. Hence, 
$\{ (b_i,m_{ic}) ,(b_j,m_{jc})\}$ is the unique pair in a block of $\mathcal{F}\circ M$ that yields the difference $(g_1,g_2)$.
It is also clear that no differences $(h,g_2)$ occur with $h \in H$, because no differences in $H$ occur in the blocks in $B \in \mathcal{F}$.
%
%

We now consider the Banff property. We need to show that the base blocks in $\mathcal{F}\circ M$ are all disjoint from $H\times G_2$, and
the base blocks and their negatives are pairwise disjoint. First, we prove that there are no points $(g_1,g_2)$ in a block of  ${\mathcal F}\circ M$ if $g_1 \in H$.
Suppose $(g_1,g_2) \in B\circ M^c$, where $g_1 \in H$ and $B \in \mathcal{F}$. Then $g_1 \in B$, which is impossible because $\mathcal{F}$ is
a $(G_1,H,k,1)$-BRDF.

Finally, denote $-(\mathcal{F}\circ M) = \{ -(B\circ M^c) : \ B\in{\cal F}, 1\leq c\leq |G_2|\}.$
We show that a point $(g_1,g_2)$ is in at most one block in $(\mathcal{F}\circ M) \cup -(\mathcal{F}\circ M)$, where 
$g_1 \in G_1 \setminus H$ and $g_2 \in G_2$.  The point $g_1$ is in at most one block of  $\mathcal{F} \cup -\mathcal{F}$,
where $-\mathcal{F} = \{ -B : B \in \mathcal{F}\}$, since $\mathcal{F}$ is
a $(G_1,H,k,1)$-BRDF. Assume $g_1 \in B \in \mathcal{F}$ (the argument  is similar if $g_1 \in -B \in -\mathcal{F}$. 
We have $g_1 = b_i \in B$ for some $i$.
There is a unique column $c$ of $M$ such that $m_{ic} = g_2$, because the difference matrix $M$ is homogeneous. So $(g_1,g_2) \in B \circ M^c$ and $(g_1,g_2)$ occurs in no other block of $(\mathcal{F}\circ M) \cup -(\mathcal{F}\circ M)$.
\end{proof}

The next lemma is just a standard ``filling in holes'' construction.

\begin{Lemma}
\label{DF.lem}
If there exists a $(G_1,H,k,1)$-BRDF with $|G_1| = v$ and $|H| = h$, and a nested $(h,k,1)$-BIBD, 
then there exists a nested $(v,k,1)$-BIBD.
\end{Lemma}

\begin{proof}
The $(G_1,H,k,1)$-BRDF generates a nested $(k,1)$-GDD of type $h^{v/h}$. Fill in each group with a nested $(h,k,1)$-BIBD.
\end{proof}

The two preceding lemmas can be used to construct  nested $(v,4,1)$-BIBDs for several values of $v$, as summarized in Table \ref{DF.tab}.
Note that the required BRDFs were shown to exist in Examples \ref{4041.exam}--\ref{12441.exam}.

\begin{table}
\caption{Applications of Lemma \ref{DF1.lem} and \ref{DF.lem}}
\label{DF.tab}
$$\begin{array}{|c|c|} \hline
$v$ &  {\rm Ingredients}\\ \hline
280 & $(40,4,4,1)-${\rm BRDF}  \ + \ $(7,4,1)-${\rm HDM} \ + \ {\rm nested} \ $(28,4,1)-${\rm BIBD} \\ \hline
364 & $(52,4,4,1)-${\rm BRDF}  \ + \ $(7,4,1)-${\rm HDM} \ + \ {\rm nested} \ $(28,4,1)-${\rm BIBD} \\ \hline
520 & $(40,4,4,1)-${\rm BRDF}  \ + \ $(13,4,1)-${\rm HDM} + \ {\rm nested} \ $(52,4,1)-${\rm BIBD} \\ \hline
532 & $(76,4,4,1)-${\rm BRDF}  \ + \ $(7,4,1)-${\rm HDM} + \ {\rm nested} \ $(28,4,1)-${\rm BIBD} \\ \hline
700 & $(100,4,4,1)-${\rm BRDF}  \ + \ $(7,4,1)-${\rm HDM} + \ {\rm nested} \ $(28,4,1)-${\rm BIBD} \\ \hline
868 & $(124,4,4,1)-${\rm BRDF}  \ + \ $(7,4,1)-${\rm HDM} + \ {\rm nested} \ $(28,4,1)-${\rm BIBD} \\ \hline
\end{array}$$
\end{table}

\bigskip

Our last theoretical construction yields nested $(v,4,1)$-BIBDs for $v = 4q$ where $q \equiv 1 \bmod 12$ is a prime or prime power, provided that certain conditions are satisfied. The next lemma is based on ideas that were introduced in the discussion preceding Example \ref{4041.exam}.

\begin{Lemma}\label{SBRDF}
Let ${\mathcal F}$ be a $(G,H,k,1)$-BRDF with $G$ a group of order $v$ and $H$ a subgroup of $G$ of order $k$.
If there exists a non-trivial coset of $H$ having empty intersection with all the base blocks of ${\mathcal F}$ and with all their negatives, 
then there exists a nested $(v,k,1)$-BIBD.
\end{Lemma}
\begin{proof}
Let $dev{\mathcal F}$ be the development of ${\mathcal F}$ and let ${\mathcal H}$ be the set of cosets
of $H$ in $G$. The pair $(G,dev{\mathcal F} \ \cup \ {\mathcal H})$ is a $(v,k,1)$-BIBD.
By assumption, there is an element $g\in G\setminus H$ such that $H+g$ is disjoint from all the base blocks of ${\mathcal F}$
and with all their negatives.
Consider the map
$$f: dev{\mathcal F} \ \cup \ {\mathcal H} \longrightarrow G$$
defined by 
\[
\begin{array}{ll}
f(B+t) =t & \text{ for all } B\in {\mathcal F} \text{ and for all } t\in G\\
f(H+t) =g+t & \text{ for all } t\in G.
\end{array}
\]
It is straightforward to check that $f$ is a nesting and then the assertion follows.
\end{proof}

\begin{Lemma}\label{4q,4,4,1}
Let $q=12n+1$ be a prime power and let $C^6$ be the subgroup of $\eff_q^*$ of index $6$.
Assume that there exists a
$16$-tuple $(a_1,a_2,\dots,a_{16})$ of elements of $\eff_q^*$ satisfying the following conditions:
\begin{itemize}
\item[(i)] Each of the four lists of differences 
\begin{align*}\Delta_{0,0} & =\{a_1-a_2, a_5-a_6,a_{10}-a_{11},a_{10}-a_{12},a_{11}-a_{12},a_{15}-a_{16}\},\\
\Delta_{0,1}&=\{a_1-a_3,a_2-a_3,a_{5}-a_{7},a_{6}-a_{7},a_{14}-a_{15},a_{14}-a_{16}\},\\
\Delta_{1,0}&=\{a_1-a_4,a_2-a_4,a_{5}-a_{8},a_{6}-a_{8},a_{13}-a_{15},a_{13}-a_{16}\},\\
\Delta_{1,1}&=\{a_3-a_4,a_7-a_8,a_{9}-a_{10},a_{9}-a_{11},a_{9}-a_{12},a_{13}-a_{14}\}
\end{align*} is a complete  system 
of representatives for the cosets of $C^6$ in $\eff_q^*$.
\item[(ii)]
Each of the four lists 
\begin{align*}
U_{0,0}&=\{a_1,a_2,a_5,a_6,a_9\},\\
U_{0,1}&=\{a_3,a_7,a_{13}\},\\
U_{1,0}&=\{a_4,a_8,a_{14}\},\\
U_{1,1}&=\{a_{10},a_{11},a_{12},a_{15},a_{16}\}
\end{align*}
is a partial  system 
of representatives for the cosets of $C^6$ in $\eff_q^*$ distinct from $C^6$.
\end{itemize}
Then there exists a  nested $(4q,4,1)$-BIBD.
\end{Lemma}

\begin{proof}
Let us consider the following 4-subsets of $\eff_4\times\eff_q$: 
\begin{align*}
B_1&=\{(0,0,a_1),(0,0,a_2),(0,1,a_3),(1,0,a_4)\},\\
B_2&=\{(0,0,a_5),(0,0,a_6),(0,1,a_7),(1,0,a_8)\},\\
B_3&=\{(0,0,a_9),(1,1,a_{10}),(1,1,a_{11}),(1,1,a_{12})\},\\
B_4&=\{(0,1,a_{13}),(1,0,a_{14}),(1,1,a_{15}),(1,1,a_{16})\}.
\end{align*}
One can patiently check that we have
$$\bigcup_{i=1}^4\Delta B_i=\bigcup_{x\in\eff_4}\{x\}\times(\Delta_x\otimes\{1,-1\}).$$
The assumption that $q=12n+1$ guarantees that $-1\in C^6$, hence $\{1,-1\}$ is a subgroup of $C^6$. 

Let $S$ be a complete system of representatives for the cosets of $\{1,-1\}$ in $C^6$ (note that $|S| = (q-1)/12$).
Then, set 
\[ {\mathcal F}=\{B_{i,s} :  1\leq i\leq 4, s\in S\},\] where $B_{i,s}$ is the set obtained from $B_i$ by multiplying the last coordinates of all its elements by $s$. Note that $|{\mathcal F}| = (q-1)/3$.

We have $\{1,-1\}\otimes S=C^6$ by definition of $S$.
Also, for every $x\in\eff_4$ we have $\Delta_x\otimes C^6=\eff_q^*$ by condition $(i)$, so we can write $(\Delta_x\otimes\{1,-1\})\otimes S=\eff_q^*$. 
It follows that we have
$$\Delta{\mathcal F}=\eff_4\times\eff_q^*=(\eff_4\times\eff_q)\setminus(\eff_4\times\{0\}),$$
which means that $\mathcal F$ is a $(\eff_4\times\eff_q,\eff_4\times\{0\},4,1)$-RDF.

The union $U_{\mathcal F}$ (counting multiplicities) of the blocks of ${\mathcal F}$ and their negatives is given by
\begin{eqnarray*}
U_{\mathcal F}&=&\bigcup_{x\in\eff_4}\{x\}\times (U_x\otimes\{1,-1\}\otimes S)\\
&=&\bigcup_{x\in\eff_4}\{x\}\times (U_x\otimes C^6).
\end{eqnarray*}
In view of $(ii)$, we have that $U_x\otimes C^6$ is the union of five or three distinct cosets of $C^6$ in $\eff_q^*$ according to 
whether $x\in\{(0,0),(1,1)\}$ or $x\in\{(0,1),(1,0)\}$, respectively.
Thus $U_{\mathcal F}$ does not have repeated elements and it does not cover any element of $\eff_4\times \{0\}$, hence
$\mathcal F$ is a BRDF. 

Finally, $U_{\mathcal F}$ is  disjoint from $\eff_4\times \{x\}$ for every $x\in C^6$.
The assertion then follows from Lemma \ref{SBRDF}.
\end{proof}

In the above construction, every block in ${\mathcal F}$ is nested with $(0,0,0)$, and the short orbit is generated from 
$H = \{(0,0,0), (0,1,0), (1,0,0), (1,1,1), (0,0,x)\}$ where the nested point is any $x \in C^6$. We only take translates of 
the short orbit through $\eff_q$. 

\begin{Example}
\label{4x13.exam}
Let us apply Lemma \ref{4q,4,4,1} in the smallest case $q=13$. Here we have $C^6=\{1,-1\}$ so that the cosets
of $C^6$ in $\eff_{13}^*$ are all pairs $\{g,-g\}$ of nonzero elements of $\eff_{13}^*$:
$$\{1,12\},\quad\quad \{2,11\},\quad\quad \{3,10\},\quad\quad\{4,9\},\quad\quad  \{5,8\},\quad\quad  \{6,7\}.$$
Let us consider the following 16-tuple 
$$(a_1,a_2,\dots,a_{16})=(2, 3, 5, 7, 4, 6, 11, 3, 5, 2, 9, 6, 3, 9, 5, 10)$$
of elements of $\eff_{13}^*$. Keeping the same notation as in Lemma \ref{4q,4,4,1} we have:
\begin{align*}
\Delta_{00}&=\{12,11,6,9,3,8\},
&\Delta_{01}&=\{10,11,6,8,4,12\},\\
\Delta_{10}&=\{8,9,1,3,11,6\},
&\Delta_{11}&=\{11,8,3,9,12,7\},\\
U_{00}&=\{2,3,4,6,5\},
&U_{01}&=\{5,11,3\},\\ 
U_{10}&=\{7,3,9\},
&U_{11}&=\{2,9,6,5,10\}.
\end{align*}
It is readily seen that conditions (i) and (ii) of Lemma \ref{4q,4,4,1} are satisfied.

In this case we can take $S=\{1\}$ and then a $(\eff_4\times\eff_{13},\eff_4\times\{0\},1)$-BRDF
is given by the following four base blocks:
\begin{align*}
B_1&=\{(0,0,2),(0,0,3),(0,1,5),(1,0,7)\},\\
B_2&=\{(0,0,4),(0,0,6),(0,1,11),(1,0,3)\},\\
B_3&=\{(0,0,5),(1,1,2),(1,1,9),(1,1,6\},\\
B_4&=\{(0,1,3),(1,0,9),(1,1,5),(1,1,10)\}.
\end{align*}
These base blocks are all nested with $(0,0,0)$ and developed through $\eff_4 \times \eff_{13}$. 
We can take $x = 1 \in C^6$, and then the short orbit is $H = \{(0,0,0), (0,1,0), (1,0,0), (1,1,1), (0,0,1)\}$, where we develop the last co-ordinate through $\eff_{13}$.
\end{Example}

\medskip

We present ``good" 16-tuples for primes $q \leq 193$, $q \equiv 1 \bmod 12$, in Table \ref{good.tab}. 
We remark that alternate constructions for $v=148$ and $v = 292$ (found by backtracking) are given in Examples \ref{14841.exam} 
 and \ref{29241.exam}, resp.

\begin{table}
\caption{``Good" 16-tuples for primes $q \leq 193$, $q \equiv 1 \bmod 12$}
\label{good.tab}
\[
\begin{array}{|r|c|} \hline
p &  (a_1,a_2,\dots,a_{16})\\ \hline
13 & (2, 3, 5, 7, 4, 6, 11, 3, 5, 2, 9, 6, 3, 9, 5, 10)\\ \hline
37 & (2, 3, 4, 5, 5, 8, 2, 14, 9, 2, 4, 23, 5, 3, 16, 24)\\ \hline
61 & (2, 4, 6, 7, 5, 8, 13, 15, 10, 2, 24, 32, 8, 25, 45, 57)\\ \hline
73 & (2, 4, 6, 7, 5, 10, 15, 26, 13, 6, 15, 44, 4, 16, 36, 43)\\ \hline
97 & (2, 4, 5, 6, 5, 10, 15, 17, 19, 2, 6, 28, 2, 53, 23, 57)\\ \hline
109 & (2, 3, 5, 6, 6, 9, 15, 17, 11, 2, 21, 58, 2, 7, 89, 98)\\ \hline
157 & (2, 3, 5, 6, 5, 13, 18, 7, 15, 2, 13, 109, 2, 40, 102, 133)\\ \hline
181 & (2, 3, 4, 7, 6, 9, 2, 16, 18, 8, 10, 41, 6, 11, 20, 169)\\ \hline 
193 & (2, 4, 6, 7, 5, 10, 15, 26, 11, 4, 5, 16, 4, 38, 80, 160)\\ \hline
\end{array}
\]
\end{table}

\bigskip

Finally, we have some results for $q$ a  prime power.

\begin{Corollary}
\label{pp.thm}
There exists a $(\eff_4\times\eff_q,\eff_4\times\{0\},4,1)$-BRDF and a nested $(4q,4,1)$-BIBD for $q = 25$, $49$ and $121$.
\end{Corollary}
\begin{proof}
These are applications of  Theorem \ref{4q,4,4,1}.
A ``good" 16-tuple in $\eff_{25} = \zed_5[x]/(x^2+x+2)$ is
\begin{multline*}(x, 1+x, 2x, 3x, 2+x, 3+2 x, 1+x, 4+4x,\\ 
3+x, 4+x, 4+3x, 2+2x, 3+x, 2+3x, x, 4+2x).
\end{multline*}
A ``good" 16-tuple in $\eff_{49}= \zed_7[x]/(x^2+1)$ is
\begin{multline*}(2x, 3x, 4x, 5x, 1+2x, 1+4x, 5x, 2+4x,\\
2+3x, 1+2x, 4, 4+3x, 2+3x, 6+2x, 4+6x, 5+4x).
 \end{multline*}
 A ``good" 16-tuple in $\eff_{121}= \zed_{11}[x]/(x^2+1)$ is
\begin{multline*}
(1+x,1+2x, 2+x, 2+2x, 1+3x, 2+x, 3+4x, 4+2x,\\
1+4x, 2+x, 3+2x, 5+10x, 1+x, 7+2x, 4+7x, 10+4x).
 \end{multline*}
 \end{proof}
 
 We remark that alternate constructions (found by backtracking) for $v=196$ and $v = 484$ are given in Examples \ref{19641.exam} 
 and \ref{48441.exam}, resp.

Using the theorem of Weil on multiplicative character sums, along the lines of \cite{BP},
it is possible to prove that a 16-tuple of elements of $\eff_q$ satisfying the
conditions of Lemma \ref{4q,4,4,1} exists whenever $q>9152353$.
Since we have found a good 16-tuple by computer for all primes $q<9152353$, we can state the following.
\begin{Theorem}
There exists a $(\eff_4\times\eff_p,\eff_4\times\{0\},4,1)$-BRDF and a nested $(4p,4,1)$-BIBD for any prime $p\equiv 1 \bmod 12$.
\end{Theorem}

A good 16-tuple can probably be found also for the  prime powers $q<9152353$. However, we have not done this computation.
 
\medskip

For $v = 256$, we note that an affine plane of order $16$ is a $(256,16,1)$-BIBD. It suffices to replace every block with a nested $(16,4,1)$-BIBD (this is a PBD-closure construction).
The case $v=688$ can be handled by a recursive construction that is described in \cite[Corollary 2.2]{St24}
(we also present base blocks for $v=688$ that were found by computer in the Appendix).
The remaining exceptions, i.e.,
\[ v \in \{ 220, 268, 
304, 316, 472, 
496, 508, 544\} ,\] 
are eliminated by using a backtracking algorithm to compute nestings of given sets of base blocks. 
These are presented in the Appendix.

Summarizing, we have our main theorem.

\begin{Theorem}
There exists a nested $(v,4,1)$-BIBD for all $v \equiv 4 \bmod 12$, $ v \geq 16$. 
\end{Theorem}



\section{Summary and conclusion}
\label{summary.sec}

The next case of interest for perfect nestings is $(v,5,2)$-BIBDs. We have found infinite classes of nested $(v,5,2)$-BIBDs and nested
$(5,2)$-GDDs. We will discuss this problem further in a forthcoming paper. 

Many variations of nestings can be considered. When $k < 2 \lambda + 1$, a nesting of a $(v,k,\lambda)$-BIBD does not exist. A relaxed definition of nesting would allow new points to be added to the BIBD, so a  $(v,k,\lambda)$-BIBD would be nested into a partial $(w,k+1,\lambda+1)$-BIBD for some $w > v$. This is the approach that is taken in \cite{AGMMRRT,PBOM} in the case $k = 2$, $\lambda = 1$. 
The natural problem to consider is to construct nestings in which $w$ is minimized.

Another possible generalization would be to study nestings of 
$(v,k,\lambda)$-BIBDs into (partial) $(v,k,\lambda')$-BIBDs, where $\lambda' > \lambda + 1$. We are not aware of any prior work on this problem.

\appendix

\section{Appendix}

\begin{Example}
\label{13641.exam}
%
%
{\rm Base blocks for a nested $(136,4,1)$-BIBD: 
\[
\begin{array}{ccc}
\{134,135,5,27,0\} \quad \quad \quad &
\{130,133,15,81,0\} \quad \quad \quad &
\{123,132,50,112,0\}\\
\{129,20,46,96,0\} \quad \quad \quad &
\{127,72,14,28,0\} \quad \quad \quad &
\{87,92,128,111,0\}\\
\{103,105,115,38,0\} \quad \quad \quad &
\{99,107,19,61,0\} \quad \quad \quad &
\{100,104,69,53,0\}\\
\{74,89,114,17,0\} \quad \quad \quad &
\{58,71,101,26,0\} \quad \quad \quad &
\{0,34,68,102,11\}
\end{array}
\]}
\end{Example}

\begin{Example}
\label{14841.exam}
{\rm Base blocks for a nested $(148,4,1)$-BIBD:
\[
\begin{array}{ccc}
\{147,111,108,6,0\} \quad \quad \quad &
\{146,82,15,105,0\} \quad \quad \quad &
\{145,97,130,68,0\}\\
\{141,143,87,119,0\} \quad \quad \quad &
\{123,139,128,50,0\} \quad \quad \quad &
\{125,137,138,24,0\}\\
\{140,112,85,55,0\} \quad \quad \quad &
\{144,102,94,14,0\} \quad \quad \quad &
\{131,91,100,110,0\}\\
\{92,136,41,67,0\} \quad \quad \quad &
\{39,35,84,122,0\} \quad \quad \quad &
\{33,27,47,99,0\}\\
\{0,37,74,111,16\} \quad \quad \quad &
\end{array}
\]}
\end{Example}

\begin{Example}
\label{16041.exam}
%
%
{\rm Base blocks for a nested $(160,4,1)$-BIBD:
\[
\begin{array}{ccc}
\{1, 2, 6, 12, 0\} \quad \quad \quad &
\{4, 7, 19, 37, 0\} \quad \quad \quad & 
\{5, 14, 50, 104, 0\} \\
\{3, 10, 32, 59, 0 \} \quad \quad \quad & 
\{8, 29, 95, 16, 0  \} \quad \quad \quad & 
\{11, 74, 112, 35, 0  \}\\
\{9, 28, 76, 119, 0  \} \quad \quad \quad  &
\{17, 31, 63, 88, 0  \} \quad \quad \quad &
\{55, 71, 133, 20, 0  \} \\
\{30, 47, 67, 135, 0  \}\quad \quad \quad &
\{13, 36, 121, 77, 0  \}\quad \quad \quad &
\{134, 136, 34, 108, 0\}\\
\{38, 69, 103, 145, 0  \} \quad \quad \quad &
\{0,40,80,120, 18\} \quad \quad \quad &
\end{array}
\]}
\end{Example}

\begin{Example}
\label{17241.exam}
%
%
{\rm Base blocks for a nested $(172,4,1)$-BIBD:
\[
\begin{array}{ccc}
\{171,132,109,62,0\} \quad \quad \quad &
\{170,71,124,137,0\} \quad \quad \quad &
\{131,103,107,128,0\}\\
\{169,150,55,160,0\} \quad \quad \quad &
\{165,130,147,96,0\} \quad \quad \quad &
\{37,81,122,39,0\}\\
\{166,46,134,146,0\} \quad \quad \quad &
\{164,93,82,67,0\} \quad \quad \quad &
\{156,157,33,149,0\}\\
\{84,4,121,10,0\} \quad \quad \quad &
\{45,113,5,83,0\} \quad \quad \quad &
\{94,99,72,153,0\}\\
\{116,161,85,145,0\} \quad \quad \quad &
\{32,68,125,18,0\} \quad \quad \quad &
\{0,43,86,129,9\}
\end{array}
\]}
\end{Example}

\begin{Example}
\label{18441.exam}
%
%
{\rm Base blocks for a nested $(184,4,1)$-BIBD:
\[
\begin{array}{ccc}
\{183,12,99,113,0\} \quad \quad \quad &
\{182,21,148,60,0\} \quad \quad \quad &
\{181,23,103,65,0\}\\
\{178,180,63,128,0\} \quad \quad \quad &
\{179,24,84,142,0\} \quad \quad \quad &
\{175,10,38,120,0\}\\
\{171,7,50,127,0\} \quad \quad \quad &
\{165,166,170,176,0\} \quad \quad \quad &
\{140,143,155,173,0\}\\
\{159,168,20,74,0\} \quad \quad \quad &
\{129,156,80,58,0\} \quad \quad \quad &
\{137,145,162,34,0\}\\
\{82,106,157,141,0\} \quad \quad \quad &
\{95,167,136,88,0\} \quad \quad \quad &
\{115,147,54,94,0\}\\
\{0,46,92,138,15\} \quad \quad \quad &
\end{array}
\]}
\end{Example}

\begin{Example}
\label{19641.exam}
%
%
{\rm Base blocks for a nested $(196,4,1)$-BIBD:
\[
\begin{array}{ccc}
\{195,15,115,135,0\} \quad \quad \quad &
\{194,8,142,19,0\} \quad \quad \quad &
\{192,14,49,118,0\}\\
\{193,11,163,45,0\} \quad \quad \quad &
\{184,191,79,51,0\} \quad \quad \quad &
\{190,17,59,123,0\}\\
\{183,189,70,96,0\} \quad \quad \quad &
\{186,22,76,116,0\} \quad \quad \quad &
\{167,168,172,9,0\}\\
\{156,159,171,74,0\} \quad \quad \quad &
\{119,128,164,69,0\} \quad \quad \quad &
\{114,141,53,160,0\}\\
\{165,50,178,107,0\} \quad \quad \quad &
\{111,158,150,133,0\} \quad \quad \quad &
\{35,176,152,101,0\}\\
\{138,169,97,140,0\} \quad \quad \quad &
\{0,49,98,147,23\} \quad \quad \quad &
\end{array}
\]}
\end{Example}

\medskip

\begin{Example}
\label{22041.exam}
%
%
{\rm Base blocks for a nested $(220,4,1)$-BIBD:
\[
\begin{array}{ccc}
\{219,13,101,78,0\} \quad \quad \quad &
\{217,17,77,145,0\} \quad \quad \quad &
\{218,6,50,121,0\}\\
\{216,18,44,139,0\} \quad \quad \quad &
\{215,11,89,175,0\} \quad \quad \quad &
\{211,16,177,57,0\}\\
\{213,26,68,137,0\} \quad \quad \quad &
\{201,212,30,54,0\} \quad \quad \quad &
\{192,193,197,210,0\}\\
\{181,184,196,15,0\} \quad \quad \quad &
\{155,164,200,97,0\} \quad \quad \quad &
\{140,167,55,186,0\}\\
\{191,52,156,109,0\} \quad \quad \quad &
\{182,189,12,76,0\} \quad \quad \quad &
\{157,178,87,59,0\}\\
\{125,188,135,51,0\} \quad \quad \quad &
\{129,98,159,127,0\} \quad \quad \quad &
\{72,199,162,66,0\}\\
\{0,55,110,165,14\} \quad \quad \quad &
\end{array}
\]}
\end{Example}

\begin{Example}
\label{26841.exam}
%
%
{\rm Base blocks for a nested $(268,4,1)$-BIBD:
\[
\begin{array}{ccc}
\{267,2,35,187,0\} \quad \quad \quad &
\{265,5,95,195,0\} \quad \quad \quad &
\{264,8,28,143,0\}\\
\{262,15,76,153,0\} \quad \quad \quad &
\{261,23,39,48,0\} \quad \quad \quad &
\{259,29,63,190,0\}\\
\{257,34,139,162,0\} \quad \quad \quad &
\{252,24,26,128,0\} \quad \quad \quad &
\{251,255,22,202,0\}\\
\{248,258,79,126,0\} \quad \quad \quad &
\{238,256,38,146,0\} \quad \quad \quad &
\{228,250,111,163,0\}\\
\{216,247,90,104,0\} \quad \quad \quad &
\{254,49,50,123,0\} \quad \quad \quad &
\{183,231,237,156,0\}\\
\{236,241,64,165,0\} \quad \quad \quad &
\{224,235,53,207,0\} \quad \quad \quad &
\{208,227,19,43,0\}\\
\{94,120,135,179,0\} \quad \quad \quad &
\{132,169,198,75,0\} \quad \quad \quad &
\{166,209,222,47,0\}\\
\{155,206,213,93,0\} \quad \quad \quad &
\{0,67,134,201,25\} \quad \quad \quad &
\end{array}
\]}
\end{Example}

\begin{Example}
\label{29241.exam}
%
%
{\rm Base blocks for a nested $(292,4,1)$-BIBD:
\[
\begin{array}{ccc}
\{291,174,127,36,0\} \quad \quad \quad &
\{290,93,21,34,0\} \quad \quad \quad &
\{289,163,101,32,0\}\\
\{284,133,14,288,0\} \quad \quad \quad &
\{282,182,70,287,0\} \quad \quad \quad &
\{283,77,45,286,0\}\\
\{280,281,137,178,0\} \quad \quad \quad &
\{272,279,103,239,0\} \quad \quad \quad &
\{259,268,270,276,0\}\\
\{275,285,145,209,0\} \quad \quad \quad &
\{235,249,273,31,0\} \quad \quad \quad &
\{262,277,62,219,0\}\\
\{255,274,96,166,0\} \quad \quad \quad &
\{245,265,82,216,0\} \quad \quad \quad &
\{243,264,35,142,0\}\\
\{212,237,267,185,0\} \quad \quad \quad &
\{227,253,269,74,0\} \quad \quad \quad &
\{238,266,42,148,0\}\\
\{213,244,46,157,0\} \quad \quad \quad &
\{115,149,242,71,0\} \quad \quad \quad &
\{141,180,226,81,0\}\\
\{131,176,229,241,0\} \quad \quad \quad &
\{68,125,192,240,0\} \quad \quad \quad &
\{122,183,254,41,0\}\\
\{0,73,146,219,29\} \quad \quad \quad &
\end{array}
\]}
\end{Example}

\begin{Example}
\label{30441.exam}
%
%
{\rm Base blocks for a nested $(304,4,1)$-BIBD:
\[
\begin{array}{ccc}
\{52,235,185,92,0\} \quad \quad \quad &
\{244,39,269,283,0\} \quad \quad \quad &
\{126,300,234,164,0\}\\
\{25,173,49,29,0\} \quad \quad \quad &
\{227,123,6,232,0\} \quad \quad \quad &
\{205,294,260,208,0\}\\
\{270,272,175,211,0\} \quad \quad \quad &
\{273,284,26,217,0\} \quad \quad \quad &
\{62,78,111,203,0\}\\
\{165,191,303,42,0\} \quad \quad \quad &
\{8,55,109,197,0\} \quad \quad \quad &
\{56,118,135,166,0\}\\
\{3,121,149,150,0\} \quad \quad \quad &
\{7,15,88,292,0\} \quad \quad \quad &
\{293,302,33,40,0\}\\
\{147,159,45,75,0\} \quad \quad \quad &
\{38,53,59,237,0\} \quad \quad \quad &
\{239,261,89,176,0\}\\
\{199,222,267,103,0\} \quad \quad \quad &
\{110,142,277,287,0\} \quad \quad \quad &
\{130,167,258,36,0\}\\
\{172,225,243,19,0\} \quad \quad \quad &
\{98,156,220,295,0\} \quad \quad \quad &
\{47,112,153,282,0\}\\
\{30,133,146,223,0\} \quad \quad \quad &
\{0,76,152,228,180\} \quad \quad \quad &
\end{array}
\]}
\end{Example}

\begin{Example}
\label{31641.exam}
%
%
{\rm Base blocks for a nested $(316,4,1)$-BIBD:
\[
\begin{array}{ccc}
\{53,54,247,88,0\} \quad \quad \quad &
\{216,81,254,231,0\} \quad \quad \quad &
\{205,102,131,76,0\}\\
\{132,213,230,307,0\} \quad \quad \quad &
\{301,110,27,60,0\} \quad \quad \quad &
\{150,23,168,137,0\}\\
\{283,128,221,218,0\} \quad \quad \quad &
\{140,209,294,67,0\} \quad \quad \quad &
\{25,190,91,84,0\}\\
\{278,207,160,73,0\} \quad \quad \quad &
\{154,259,284,21,0\} \quad \quad \quad &
\{178,223,8,121,0\}\\
\{90,99,296,129,0\} \quad \quad \quad &
\{115,164,113,222,0\} \quad \quad \quad &
\{314,19,268,89,0\}\\
\{275,56,136,80,0\} \quad \quad \quad &
\{120,77,237,125,0\} \quad \quad \quad &
\{24,108,144,13,0\}\\
\{74,250,135,304,0\} \quad \quad \quad &
\{124,64,105,42,0\} \quad \quad \quad &
\{55,255,92,65,0\}\\
\{14,18,106,234,0\} \quad \quad \quad &
\{170,182,130,198,0\} \quad \quad \quad &
\{47,299,155,3,0\}\\
\{167,31,199,271,0\} \quad \quad \quad &
\{163,171,39,177,0\} \quad \quad \quad &
\{0,237,158,79,44\}
\end{array}
\]}
\end{Example}

\begin{Example}
\label{47241.exam}
%
%
{\rm Base blocks for a nested $(472,4,1)$-BIBD:
\[
\begin{array}{ccc}
\{470,471,56,71,0\} \quad \quad \quad &
\{469,8,64,368,0\} \quad \quad \quad &
\{468,15,100,217,0\}\\
\{467,25,43,70,0\} \quad \quad \quad &
\{463,465,237,437,0\} \quad \quad \quad &
\{466,10,90,293,0\}\\
\{440,461,119,172,0\} \quad \quad \quad &
\{460,21,227,369,0\} \quad \quad \quad &
\{459,31,86,181,0\}\\
\{449,458,122,174,0\} \quad \quad \quad &
\{438,455,89,109,0\} \quad \quad \quad &
\{432,456,175,391,0\}\\
\{420,454,113,272,0\} \quad \quad \quad &
\{413,453,120,186,0\} \quad \quad \quad &
\{450,24,75,233,0\}\\
\{446,28,126,136,0\} \quad \quad \quad &
\{443,50,128,302,0\} \quad \quad \quad &
\{452,66,221,303,0\}\\
\{445,78,168,279,0\} \quad \quad \quad &
\{442,84,155,322,0\} \quad \quad \quad &
\{430,83,118,295,0\}\\
\{378,191,114,439,0\} \quad \quad \quad &
\{374,55,411,433,0\} \quad \quad \quad &
\{375,176,73,435,0\}\\
\{400,157,436,404,0\} \quad \quad \quad &
\{421,257,76,426,0\} \quad \quad \quad &
\{277,414,364,280,0\}\\
\{339,153,77,403,0\} \quad \quad \quad &
\{366,48,405,428,0\} \quad \quad \quad &
\{362,164,62,425,0\}\\
\{385,143,423,392,0\} \quad \quad \quad &
\{305,142,434,313,0\} \quad \quad \quad &
\{260,398,349,266,0\}\\
\{357,173,99,427,0\} \quad \quad \quad &
\{275,431,318,343,0\} \quad \quad \quad &
\{111,387,287,180,0\}\\
\{367,127,409,380,0\} \quad \quad \quad &
\{218,57,351,232,0\} \quad \quad \quad &
\{278,418,371,290,0\}\\
\{0,118,236,354,53\} \quad \quad \quad &
\end{array}
\]}
\end{Example}

\begin{Example}
\label{48441.exam}
%
%
{\rm Base blocks for a nested $(484,4,1)$-BIBD:
\[
\begin{array}{ccc}
\{140,431,352,201,0\} \quad \quad \quad &
\{389,62,426,448,0\} \quad \quad \quad &
\{386,179,74,446,0\}\\
\{304,55,340,308,0\} \quad \quad \quad &
\{475,309,124,480,0\} \quad \quad \quad &
\{223,366,314,226,0\}\\
\{93,385,307,157,0\} \quad \quad \quad &
\{159,317,198,221,0\} \quad \quad \quad &
\{145,423,319,208,0\}\\
\{381,133,419,388,0\} \quad \quad \quad &
\{375,210,26,383,0\} \quad \quad \quad &
\{63,207,156,69,0\}\\
\{2,296,220,72,0\} \quad \quad \quad &
\{173,333,216,241,0\} \quad \quad \quad &
\{372,168,66,441,0\}\\
\{398,152,440,411,0\} \quad \quad \quad &
\{8,329,147,22,0\} \quad \quad \quad &
\{467,129,80,479,0\}\\
\{449,450,460,393,0\} \quad \quad \quad &
\{466,468,114,451,0\} \quad \quad \quad &
\{182,191,215,292,0\}\\
\{287,303,361,387,0\} \quad \quad \quad &
\{394,105,126,153,0\} \quad \quad \quad &
\{483,29,178,212,0\}\\
\{200,240,356,14,0\} \quad \quad \quad &
\{163,300,371,48,0\} \quad \quad \quad &
\{442,41,222,336,0\}\\
\{81,146,291,471,0\} \quad \quad \quad &
\{433,52,102,228,0\} \quad \quad \quad &
\{125,205,294,106,0\}\\
\{324,377,405,154,0\} \quad \quad \quad &
\{236,326,367,439,0\} \quad \quad \quad &
\{143,238,345,400,0\}\\
\{282,348,368,85,0\} \quad \quad \quad &
\{54,172,218,465,0\} \quad \quad \quad &
\{342,464,25,219,0\}\\
\{257,301,425,82,0\} \quad \quad \quad &
\{47,199,234,281,0\} \quad \quad \quad &
\{335,6,60,135,0\}\\
\{211,229,365,473,0\} \quad \quad \quad &
\{0,121,242,363,131\} \quad \quad \quad &
\end{array}
\]}
\end{Example}

\begin{Example}
\label{49641.exam}
%
%
{\rm Base blocks for a nested $(496,4,1)$-BIBD:
\[
\begin{array}{ccc}
\{85,385,304,147,0\} \quad \quad \quad &
\{290,450,328,350,0\} \quad \quad \quad &
\{364,153,45,425,0\}\\
\{100,341,137,104,0\} \quad \quad \quad &
\{386,215,24,391,0\} \quad \quad \quad &
\{470,118,67,473,0\}\\
\{44,345,265,109,0\} \quad \quad \quad &
\{13,174,53,76,0\} \quad \quad \quad &
\{108,394,287,172,0\}\\
\{157,399,196,164,0\} \quad \quad \quad &
\{3,329,139,11,0\} \quad \quad \quad &
\{84,229,179,90,0\}\\
\{457,264,186,32,0\} \quad \quad \quad &
\{245,408,289,314,0\} \quad \quad \quad &
\{356,148,43,426,0\}\\
\{187,431,230,200,0\} \quad \quad \quad &
\{128,456,268,142,0\} \quad \quad \quad &
\{427,78,30,439,0\}\\
\{131,133,401,47,0\} \quad \quad \quad &
\{492,129,144,145,0\} \quad \quad \quad &
\{448,458,18,124,0\}\\
\{395,413,444,468,0\} \quad \quad \quad &
\{379,398,475,33,0\} \quad \quad \quad &
\{237,257,446,119,0\}\\
\{93,114,190,437,0\} \quad \quad \quad &
\{462,1,152,181,0\} \quad \quad \quad &
\{115,160,313,360,0\}\\
\{180,252,459,55,0\} \quad \quad \quad &
\{22,74,165,288,0\} \quad \quad \quad &
\{445,454,143,211,0\}\\
\{476,8,54,112,0\} \quad \quad \quad &
\{312,346,113,154,0\} \quad \quad \quad &
\{60,96,197,280,0\}\\
\{240,282,66,123,0\} \quad \quad \quad &
\{486,75,92,321,0\} \quad \quad \quad &
\{246,305,358,469,0\}\\
\{233,308,489,103,0\} \quad \quad \quad &
\{218,297,424,130,0\} \quad \quad \quad &
\{169,272,283,438,0\}\\
\{81,201,227,14,0\} \quad \quad \quad &
\{241,416,2,29,0\} \quad \quad \quad &
\{0,124,248,372,149\}\\
\end{array}
\]}
\end{Example}

\begin{Example}
\label{50841.exam}
%
%
{\rm Base blocks for a nested $(508,4,1)$-BIBD:
\[
\begin{array}{ccc}
\{86,390,307,147,0\} \quad \quad \quad &
\{471,127,1,23,0\} \quad \quad \quad &
\{506,287,177,57,0\}\\
\{77,323,114,81,0\} \quad \quad \quad &
\{276,101,413,281,0\} \quad \quad \quad &
\{469,110,54,472,0\}\\
\{224,21,447,288,0\} \quad \quad \quad &
\{53,218,93,116,0\} \quad \quad \quad &
\{308,90,489,370,0\}\\
\{195,442,234,202,0\} \quad \quad \quad &
\{33,367,172,41,0\} \quad \quad \quad &
\{385,27,480,391,0\}\\
\{55,362,282,125,0\} \quad \quad \quad &
\{145,312,189,214,0\} \quad \quad \quad &
\{5,297,190,73,0\}\\
\{443,184,486,456,0\} \quad \quad \quad &
\{15,351,158,29,0\} \quad \quad \quad &
\{420,64,11,432,0\}\\
\{6,314,235,79,0\} \quad \quad \quad &
\{206,374,252,278,0\} \quad \quad \quad &
\{495,280,174,58,0\}\\
\{309,51,354,325,0\} \quad \quad \quad &
\{251,80,396,268,0\} \quad \quad \quad &
\{311,464,412,326,0\}\\
\{104,105,357,69,0\} \quad \quad \quad &
\{159,161,264,478,0\} \quad \quad \quad &
\{40,50,115,375,0\}\\
\{187,198,279,328,0\} \quad \quad \quad &
\{38,47,135,329,0\} \quad \quad \quad &
\{94,113,338,389,0\}\\
\{341,388,488,143,0\} \quad \quad \quad &
\{286,352,437,168,0\} \quad \quad \quad &
\{462,68,102,301,0\}\\
\{500,12,289,366,0\} \quad \quad \quad &
\{463,3,31,193,0\} \quad \quad \quad &
\{164,248,275,106,0\}\\
\{501,129,203,380,0\} \quad \quad \quad &
\{231,249,343,364,0\} \quad \quad \quad &
\{292,333,400,242,0\}\\
\{241,295,419,217,0\} \quad \quad \quad &
\{304,417,459,247,0\} \quad \quad \quad &
\{109,285,316,482,0\}\\
\{0,127,254,381,212\} \quad \quad \quad &
\end{array}
\]}
\end{Example}

\begin{Example}
\label{54441.exam}
%
%
{\rm Base blocks for a nested $(544,4,1)$-BIBD:
\[
\begin{array}{ccc}
\{93,421,332,162,0\} \quad \quad \quad &
\{381,13,423,448,0\} \quad \quad \quad &
\{209,523,404,277,0\}\\
\{504,224,1,508,0\} \quad \quad \quad &
\{248,59,394,253,0\} \quad \quad \quad &
\{151,309,251,154,0\}\\
\{370,155,67,442,0\} \quad \quad \quad &
\{66,243,110,136,0\} \quad \quad \quad &
\{372,143,25,443,0\}\\
\{464,185,507,471,0\} \quad \quad \quad &
\{49,405,197,57,0\} \quad \quad \quad &
\{304,463,406,310,0\}\\
\{34,365,279,112,0\} \quad \quad \quad &
\{247,426,295,323,0\} \quad \quad \quad &
\{213,530,414,290,0\}\\
\{516,239,19,529,0\} \quad \quad \quad &
\{374,188,526,388,0\} \quad \quad \quad &
\{269,430,375,281,0\}\\
\{77,409,324,158,0\} \quad \quad \quad &
\{461,97,511,540,0\} \quad \quad \quad &
\{542,316,201,78,0\}\\
\{250,518,299,266,0\} \quad \quad \quad &
\{214,29,368,231,0\} \quad \quad \quad &
\{539,157,103,10,0\}\\
\{146,147,339,128,0\} \quad \quad \quad &
\{186,195,259,261,0\} \quad \quad \quad &
\{458,468,145,180,0\}\\
\{532,9,216,403,0\} \quad \quad \quad &
\{62,84,357,489,0\} \quad \quad \quad &
\{271,298,385,260,0\}\\
\{257,288,341,505,0\} \quad \quad \quad &
\{315,355,506,54,0\} \quad \quad \quad &
\{60,120,202,225,0\}\\
\{217,307,450,89,0\} \quad \quad \quad &
\{377,488,533,88,0\} \quad \quad \quad &
\{469,501,199,219,0\}\\
\{241,292,473,119,0\} \quad \quad \quad &
\{61,122,152,492,0\} \quad \quad \quad &
\{465,16,184,230,0\}\\
\{51,171,233,396,0\} \quad \quad \quad &
\{24,63,87,208,0\} \quad \quad \quad &
\{8,64,168,242,0\}\\
\{99,164,274,472,0\} \quad \quad \quad &
\{317,415,527,258,0\} \quad \quad \quad &
\{437,65,117,200,0\}\\
\{0,136,272,408,318\} \quad \quad \quad &
\end{array}
\]}
\end{Example}

\begin{Example}
\label{68841.exam}
%
%
{\rm Base blocks for a nested $(688,4,1)$-BIBD:
\[
\begin{array}{ccc}
\{39,255,340,614,0\} \quad \quad \quad &
\{200,339,370,422,0\} \quad \quad \quad &
\{91,242,403,487,0\}\\
\{37,84,88,442,0\} \quad \quad \quad &
\{103,368,547,552,0\} \quad \quad \quad &
\{428,502,625,628,0\}\\
\{142,230,503,615,0\} \quad \quad \quad &
\{301,438,470,524,0\} \quad \quad \quad &
\{257,407,567,654,0\}\\
\{123,169,176,529,0\} \quad \quad \quad &
\{108,116,354,618,0\} \quad \quad \quad &
\{343,416,538,544,0\}\\
\{16,127,341,432,0\} \quad \quad \quad &
\{20,76,540,675,0\} \quad \quad \quad &
\{112,202,492,641,0\}\\
\{65,110,120,472,0\} \quad \quad \quad &
\{36,213,224,461,0\} \quad \quad \quad &
\{189,261,382,391,0\}\\
\{25,134,346,443,0\} \quad \quad \quad &
\{451,582,617,677,0\} \quad \quad \quad &
\{253,400,557,653,0\}\\
\{24,40,390,669,0\} \quad \quad \quad &
\{96,357,532,549,0\} \quad \quad \quad &
\{332,402,521,536,0\}\\
\{7,276,384,595,0\} \quad \quad \quad &
\{198,327,363,425,0\} \quad \quad \quad &
\{41,328,474,630,0\}\\
\{178,220,239,588,0\} \quad \quad \quad &
\{99,273,293,527,0\} \quad \quad \quad &
\{308,377,495,513,0\}\\
\{92,195,463,570,0\} \quad \quad \quad &
\{254,381,418,482,0\} \quad \quad \quad &
\{62,217,319,605,0\}\\
\{18,59,81,429,0\} \quad \quad \quad &
\{210,469,642,665,0\} \quad \quad \quad &
\{105,126,608,676,0\}\\
\{2,3,27,448,0\} \quad \quad \quad &
\{564,566,646,658,0\} \quad \quad \quad &
\{221,265,545,558,0\}\\
\{51,78,241,477,0\} \quad \quad \quad &
\{111,456,484,632,0\} \quad \quad \quad &
\{248,358,392,406,0\}\\
\{191,399,466,505,0\} \quad \quad \quad &
\{162,212,543,659,0\} \quad \quad \quad &
\{43,119,598,627,0\}\\
\{223,404,518,680,0\} \quad \quad \quad &
\{192,317,374,515,0\} \quad \quad \quad &
\{107,133,439,479,0\}\\
\{64,129,283,523,0\} \quad \quad \quad &
\{264,302,379,522,0\} \quad \quad \quad &
\{53,350,408,621,0\}\\
\{229,436,514,666,0\} \quad \quad \quad &
\{149,179,228,684,0\} \quad \quad \quad &
\{158,251,383,454,0\}\\
\{95,235,316,533,0\} \quad \quad \quad &
\{57,199,323,398,0\} \quad \quad \quad &
\{146,445,550,616,0\}\\
\{0,172,334,516,419\} \quad \quad \quad &
\end{array}
\]}
\end{Example}


\end{document}